\documentclass{SCAEOL}
\usepackage{graphicx,amssymb,mathrsfs,amsmath,color,fancyhdr}
\usepackage{amsfonts}
\usepackage{graphics}
\usepackage{float}
\usepackage{multirow}
\usepackage{amsmath}
\usepackage{amssymb}
\usepackage{psfrag}
\usepackage{amsmath}
\usepackage{graphicx}
\usepackage{enumerate}
\numberwithin{equation}{section}
\newcommand{\mF}{ {\mathcal{F}} }

\newcommand{\ds}{\displaystyle}
\newcommand{\be}{\begin{equation}}
\newcommand{\ee}{\end{equation}}
\newtheorem{sch}{ {\bf Scheme} }[section]

\newtheorem{lem}{Lemma}[subsection]

\newtheorem{thm}{Theorem}[section]
\newtheorem{assum}{Assumption}[section]
\newtheorem{pro}{Proposition}[section]

\newtheorem{rem}{Remark}[section]

\begin{document}

\Year{2013} %
\Month{January}
\Vol{56} %
\No{1} %
\BeginPage{1} %
\EndPage{XX} %

\title{Convergence Error Estimates of the Crank-Nicolson Scheme for Solving Decoupled FBSDEs
}{}


\author[1]{Yang Li}{}
\author[2]{Jie Yang}{}
\author[2]{Weidong Zhao}{Corresponding author}

\address[{\rm1}]{College of Science, University of Shanghai for Science and Technology, Shanghai {\rm 200093},  China;}
\address[{\rm2}]{School of Mathematics, Shandong University, Jinan,
Shandong {\rm 250100}, China}\vspace{3mm}
\Emails{yangli@usst.edu.cn,
yangjie218@mail.sdu.edu.cn, wdzhao@sdu.edu.cn}\maketitle


 {\begin{center}
\parbox{14.5cm}{\begin{abstract}
 The Crank-Nicolson (short for C-N) scheme
for solving {\it backward stochastic differential equation} (BSDE),
driven by Brownian motions,
was first developed by the authors W. Zhao, L. Chen and S. Peng
[SIAM J. Sci. Comput., 28 (2006), 1563--1581],
and numerical experiments showed
that the accuracy of this C-N scheme was of second order for solving BSDE.
This C-N scheme was extended to solve decoupled
{\it forward-backward stochastic differential equations} (FBSDEs)
by  W. Zhao, Y. Li and Y. Fu [Sci. China. Math., 57 (2014), 665--686],
and it was numerically shown that the accuracy of the extended C-N scheme
was also of second order.

To our best knowledge,
among all one-step (two-time level) numerical schemes with second-order accuracy
for solving BSDE or FBSDEs, such as the ones in the above two papers
and the one developed by the authors
D. Crisan and K. Manolarakis [Ann. Appl. Probab., 24, 2 (2014), 652--678],
the C-N scheme is the simplest one in applications.
The theoretical proofs of second-order error estimates
reported in the literature for these schemes for solving decoupled FBSDEs
did not include the C-N scheme.

The purpose of this work is to theoretically analyze the error estimate of
the C-N scheme for solving decoupled FBSDEs.
Based on the Taylor and It\^o-Taylor expansions, the Malliavin calculus theory
(e.g., the multiple Malliavin integration-by-parts formula),
and our new truncation error cancelation techniques,
we rigorously prove  that the strong convergence rate of the C-N scheme
is of second order for solving decoupled FBSDEs,
which fills the gap between the second-order numerical and theoretical analysis
of the C-N scheme.\vspace{-3mm}
\end{abstract}}\end{center}}




\baselineskip 11pt\parindent=10.8pt  \wuhao
\section{Introduction}
Let $(\Omega,\mathcal{F}, \mathbb F, P)$ be a filtered complete probability space,
where $\mathbb F = (\mathcal F_t)_{0\le t\le T}$ is the natural filtration of
the standard $d$-dimensional Brownian motion $W_t = (W_t^1,\ldots,W_t^d)^\top$,
$t\in [0,T]$,
on the probability space $(\Omega, \mathcal{F}, \mathbb F, P)$,
and $T$ is a fixed finite horizon.
Let $L^2=L^2_{\mathcal{F}}(0,T)$ be the set of all
$\mathcal{F}_t$-adapted and mean-square-integrable vector or matrix
processes for $t\in[0,T]$.

In this paper, on the space $(\Omega,\mathcal{F}, \mathbb F, P)$,
we consider numerical solutions of decoupled
{\it forward-backward stochastic differential equations} (FBSDEs)
in the following integral form.
\begin{equation}\label{DFBSDEs1}
\left\{
\begin{aligned}
X_t=\;& X_0+\int_0^t b(s,X_s) ds+\int_0^t\sigma(s,X_s)dW_s, & \text{(SDE)} \\
Y_t=\;& \varphi(X_T)+\int_t^Tf(s,X_s,Y_s,Z_s)ds-\int_t^TZ_sdW_s,  & \quad\text{(BSDE)}
\end{aligned}
\right.
\end{equation}
for $t\in [0,T]$, where $X_0$ is the initial condition of the {\it forward stochastic
differential equation} (SDE), $\varphi(X_{T})$ is the terminal condition of the
{\it backward stochastic differential equation} (BSDE),
$b$ is the drift coefficient valued in $\mathbb{R}^d$,
$\sigma$ is the diffusion matrix valued in $\mathbb{R}^{d\times d}$,
and $f$ valued in $\mathbb{R}$ is the generator function.
Note that the two integrals with respect to $W_s$ in \eqref{DFBSDEs1} are the It\^o-type integrals.

A triple $(X_s, Y_s, Z_s): [0,T]\times \Omega \rightarrow
\mathbb{R}^d\times \mathbb{R}\times \mathbb{R}^{1\times d}$ is called an
$L^2$-adapted solution of  \eqref{DFBSDEs1} if it is
$\mathcal{F}_s$-adapted, $L^2$-integrable, and satisfies
(\ref{DFBSDEs1}).
In \cite{PP90}, under some standard conditions on the coefficients of \eqref{DFBSDEs1},
Pardoux and Peng originally proved the existence and uniqueness of the solution
of nonlinear BSDE with more general terminal condition $Y_{T}=\xi\in \mathcal F_{T}$.
And the solution $\big(Y_s, Z_s\big)$ of \eqref{DFBSDEs1} can be represented as (\cite{Evan98,KPQ97,LSU68,PR14,Peng91})
\begin{equation}\label{s2:e2}
Y_s = u(s,X_s), \quad  Z_s=u_x(s,X_s)\sigma(s,X_s),\quad \forall\; s\in [0,T),
\end{equation}
where
$u(t,x)$ is the smooth solution of the following parabolic
 partial differential equation~(PDE).
\begin{equation}\label{PDEs}
\begin{array}{rl}
u_t(t,x)
+\frac{1}{2}\sum\limits_{i,j=1}^d[\sigma\sigma^*]_{i,j}(t,x) u_{x_ix_j}(t, x)
+\sum\limits_{i=1}^db_i(t,x) u_{x_i}(t, x)
+f(t,x,u(t,x),u_x(t,x)\sigma(t,x))& \!\!\!\!=0
\end{array}
\end{equation}
with the terminal condition $u(T,x)=\varphi(x)$.

FBSDEs have important applications in many fields including mathematical
finance, partial differential equations, stochastic control, risk measure,
and so on \cite{An93,KPQ97,MY99,Peng90,Peng97}.
So it is interesting and important to find solutions of FBSDEs.
Usually, it is difficult to get the analytical solutions in an explicit closed form.
Thus numerical methods for solving FBSDEs are desired,
especially accurate, effective and efficient ones.
Many numerical schemes for solving BSDE and decoupled FBSDEs have been developed,
among which some are Euler-type methods with convergence rate $\frac{1}{2}$, such as
\cite{BD07,BZ08,BT04,CZ05,DM06,DMP96,GL07,GLW05,MPST02,Zhang04} and
some are high-order numerical methods, such as
\cite{CM14,MSZ08,ZCP06,ZFZ14,ZLZ12,ZLJ13,ZLF14,ZWP09,ZZJ10,ZZJ14}.

To our best knowledge in the literature, up to now,
one-step second-order numerical schemes
for solving BSDE and decoupled FBSDEs were proposed and studied in
\cite{ZCP06,ZLJ13,ZLF14,ZZJ14,CM14}.
In 2006, Zhao, Chen and Peng proposed numerical schemes for solving BSDE in \cite{ZCP06},
in which the Crank-Nicolson (short for C-N) is included. Numerical
experiments showed that the accuracy of the C-N scheme was of second order
for solving BSDE and its second-order convergence
was theoretically proved in \cite{ZLJ13}. 
And in 2014, Zhao, Li and Fu proposed three one-step second-order schemes,
including the C-N scheme, for solving decoupled FBSDEs \cite{ZLF14},
and theoretically proved second-order convergence of them
but not of the C-N one.
By introducing new Gaussian processes,
second-order numerical schemes
were presented and analyzed for solving BSDE \cite{CM14}
and for decoupled FBSDEs in \cite{ZZJ14}.
The introduced new Gaussian processes simplified
the proof of error estimates of the schemes, but doubled the computational complexity
for solving BSDE or FBSDEs.

Among all these one-step second-order schemes,
concerning their applications and coding in solving BSDE or FBSDEs,
the simplest one is the C-N scheme.
It was proposed in \cite{ZCP06} for solving BSDE
and the extension for solving decoupled FBSDEs was introduced in \cite{ZLF14}.
The second-order convergence rate of the C-N scheme for BSDE
was proved in \cite{ZLJ13},
but for decoupled FBSDEs is still open until now.

The purpose of this paper is to give a rigorously theoretical analysis on second-order convergence
of the C-N scheme for solving decoupled FBSDEs \eqref{DFBSDEs1}.
Compared with the proof in \cite{ZLJ13} for BSDE,
the analysis for decoupled FBSDEs is much more difficult and complex.
By the Taylor and It\^o-Taylor expansions, the theory of multiple Malliavin calculus,
and the error cancelation techniques, we  are able to rigorously prove
a general error estimate result for the C-N scheme, 
and based on this result, we finally obtained
the theoretical second-order error estimate of the scheme
for solving the decoupled FBSDEs.

Some notation to be used:

\begin{itemize}
\item
$A^\top$: the transpose of vector or matrix $A$.
\item
$|\cdot|$: the norm for vector or matrix defined by $|A|^2=$trace($A^\top A$).
\item
$C_b^{l, k, k,k}$: the set of continuously
 differentiable functions $\psi:[0,T]\times \mathbb{R}^d\times \mathbb{R}\times \mathbb{R}^{
d}\rightarrow \mathbb{R}$ with uniformly bounded partial derivatives
$\partial^{l_1}_t\psi$ and $
\partial^{k_1}_x\partial^{k_2}_y\partial^{k_3}_z\psi$ for $\frac{1}{2}\leq l_1\leq l$ and
$1\leq k_1+k_2+k_3\leq k$. Analogously we define $C^{l,k,k}_b$ and $C^{l,k}_b$.
\item
$C^{k}_b$: the set of functions $\psi: x\in
\mathbb{R}^d\rightarrow \mathbb{R}$ with uniformly bounded partial derivatives
$\partial^{k_1}_x\psi$ for $1\leq k_1\leq k$.
\item
$\mF_{s}^{t,x}(t\leq s\leq T)$: the
$\sigma$-field generated by the diffusion process
$\{X_{r},t\leq r\leq s, X_t = x\}$.
\item
$\mathbb{E}_{s}^{t,x}[\eta]$: the conditional
mathematical expectation of the random variable $\eta$ under the
$\sigma$-field $\mF_{s}^{t,x}$, i.e.,
$\mathbb{E}_{s}^{t,x}[\eta]=\mathbb{E}[\eta|\mF_{s}^{t,x}]$. Let
$\mathbb{E}_t^x[\eta]=\mathbb{E}[\eta|\mF_{t}^{t,x}]$.
\item
$\partial_x\psi$: the matrix valued function
$\partial_x\psi=(\partial_{x^j}\psi^i)_{d\times d}$ $(1\leq i\leq d,
1\leq j\leq d)$ for vector function
$\psi=(\psi^1,\dots,\psi^d)^{\top}$.

\end{itemize}

The rest of the paper is organized as follows. After we
introduce some preliminaries in Section 2,
we review the C-N scheme proposed in
\cite{ZLF14} for solving FBSDEs \eqref{DFBSDEs1} in Section 3.
Then we state our main error estimate results
for the C-N scheme in Section 4, and prove
them in Section 5. In Section 6, some conclusions are given.
\section{ Preliminaries}
\subsection{Variational equations of the decoupled FBSDEs }
Let $(X_r^{t,x},Y_r^{t,x},Z_r^{t,x})$ be the solution of the FBSDEs
\begin{equation}\label{DFBSDEs2}
\left\{
\begin{aligned}
X_r^{t,x}=  \;&
x+\int_t^r b(s,X_s^{t,x}) ds+\int_t^r\sigma(s,X_s^{t,x})dW_s, & \text{(SDE)} \\
Y_r^{t,x}=  \;&
\varphi(X_T^{t,x})+\int_r^Tf(s,X_s^{t,x},Y_s^{t,x},Z_s^{t,x})ds-\int_r^T Z_s^{t,x} dW_s,  & ~~\text{(BSDE)}
\end{aligned}
\right.
\end{equation}
for $r\in [t,T]$. Here the superscript $^{t,x}$ means that the forward SDE starts
from time $t$ at space point $x$.

Let $\nabla_{x_i}  X_r^{t,x}$  and $\nabla_{x_i} Y_r^{t,x}$  be respectively the
variation of $X_r^{t,x}$ and $Y_r^{t,x}$ with respect to (w.r.t.) $x_i$ which is the
$i$-th component of $x=(x_1,\ldots,x_d)^\top$.
Taking variation $\nabla_{x_i}$ on both sides of the equations in \eqref{DFBSDEs2},
we deduce
\begin{equation}\label{s2:e1}
\left\{ \begin{aligned}
\nabla_{x_i} X_r^{t,x}=\;&
e_i+\int_t^rb_x(s,X_s^{t,x})\nabla_{x_i} X_s^{t,x}ds+\int_t^r\sigma_x^j(s,X_s^{t,x})\nabla_{x_i} X_s^{t,x}\, dW_s, &  \\
\nabla_{x_i} Y_r^{t,x}=\;&
\varphi_x(X_T^{t,x})\nabla_{x_i} X_T^{t,x}+\int_r^T\nabla_{x_i} f(s,X_s^{t,x},Y_s^{t,x},Z_s^{t,x})ds-\int_r^T\nabla_{x_i} Z_s^{t,x}dW_s,  &
\end{aligned}
\right.
\end{equation}
where $e_i=\overbrace{(0,\ldots,0,1,0,\ldots,0)}^i$ is the $i$-th coordinate basis
vector of $\mathbb{R}^d$, $\sigma^j$ is the $j$-th column of $\sigma(\cdot)$, and
\begin{equation*}\begin{array}{rl}
\nabla_x X_s^{t,x}=&\!\!\!\! [\nabla_{x_1} X_s^{t,x},\ldots,\nabla_{x_d} X_s^{t,x}]_{d\times d}, \quad
\nabla_x Y_s^{t,x}= [\nabla_{x_1} Y_s^{t,x},\ldots, \nabla_{x_d} Y_s^{t,x}]_{1\times d},  \\
\nabla_x Z_s^{t,x}=&\!\!\!\! \big[[\nabla_{x_1} Z_s^{t,x}]^\top,\ldots,[\nabla_{x_d} Z_s^{t,x}]^\top\big]_{d\times d},\\
\nabla_{x_i} f(s,X_s^{t,x},Y_s^{t,x},Z_s^{t,x})=&\!\!\!\!
f_x(s,X_s^{t,x},Y_s^{t,x},Z_s^{t,x})\nabla_{x_i} X_s^{t,x}
 +f_y(s,X_s^{t,x},Y_s^{t,x},Z_s^{t,x})\nabla_{x_i} Y_s^{t,x}\\
&\!\! +f_z(s,X_s^{t,x},Y_s^{t,x},Z_s^{t,x})\nabla_{x_i} Z_s^{t,x}.
\end{array}\end{equation*}

\subsection{The It\^o-Taylor scheme for  forward SDE}\label{sec2}

For the time interval $[0,T]$, we first introduce the following  time partition:
\begin{equation*}\label{tpart}
0=t_0<\dots<t_{N-1}<t_N=T
\end{equation*}
with $\Delta = t_{n+1}-t_n$ for $n = 0,1,\dots,N-2$, and $t_N-t_{N-1} = \Delta^2$.

We shall call a row vector $\alpha=(j_1,j_2,\ldots,j_l)$ with $j_i\in \{0,1,\ldots,d\}$ for $i\in \{1,2,\ldots,l\}$,
a multi-index of length $l:=l(\alpha)\in \{1,2,\ldots,d\}$,
and  denote by $v$ the multi-index of length zero $(l(v):=0)$.
Let $\mathcal{M}$ be the set of all multi-indices, that is,
\begin{equation*}\label{m}
\mathcal{M}=\Big\{(j_1,j_2,\ldots,j_l): j_i\in \{0,1,\ldots,d\}, i\in\{1,2, \ldots,l\} ~~
\textrm{for} ~~ l=1,2,\ldots\Big\}\cup \{v\}.
\end{equation*}
Given a multi-index $\alpha\in \mathcal{M}$ with $l(\alpha)\geq 1$,
we write $-\alpha$ and $\alpha-$ for
the multi-index in $\mathcal{M}$ by deleting the first and last component of $\alpha$, respectively.
Denote by $I_\alpha[g_\alpha(\cdot)]_{t_n,t_{n+1}}$ the multiple It\^o integral
recursively defined by
\begin{equation*}
I_\alpha[g_\alpha(\cdot)]_{t_n,t_{n+1}}=\begin{cases}
 X^n, &l=0,\\
\int_{t_n}^{t_{n+1}}I_{\alpha-}[g_\alpha(\cdot)]_{t_n,s}ds, &l\geq 1,\; j_l=0,\\
\int_{t_n}^{t_{n+1}}I_{\alpha-}[g_\alpha(\cdot)]_{t_n,s}dW_s^{j_l}, &l \geq 1,\; j_l\geq 1,
\end{cases}
\end{equation*}
where
the It\^o coefficient functions $g_\alpha(t,x)$ are defined by
\begin{equation*}
g_\alpha(t,x)=\begin{cases}
 x,  &l=0,\\
g_{(0)}=b(t,x),\;\, g_{(1)}=\sigma(t,x), &l=1,\\
L^{j_1}g_{-\alpha}, &l > 1,
\end{cases}
\end{equation*}
 for all $(t,x)\in \mathbb{R}\times \mathbb{R}^d$,
   and $L^{j}$ are the differential operators defined by
\begin{equation}\label{defL}
L^0=\partial_t+ \sum\limits_{k=1}^d b_k\frac{\partial}{\partial x_k}+\frac{1}{2}\sum\limits_{k,l=1}^d\sum\limits_{j=1}^d\sigma_{kj}\sigma_{lj}\frac{\partial^2}{\partial x_k\partial x_l}; \quad L^j=\sum\limits_{i=1}^d\sigma_{ij}\frac{\partial}{\partial x_i},\quad\; 1\leq j\leq d.
\end{equation}

In this paper, we will use the following weak order-2 It\^o-Taylor
schemes for solving SDE:
\begin{equation}\label{Xnn10}
X^{n+1}
=\sum_{\alpha\in \Gamma_2}g_\alpha(t_n,X^n)I_{\alpha,n}
=X^n+\phi^n,
\end{equation}
where $X^{n}=\big(X_1^{n}, \ldots, X_d^{n}\big)^\top$, $\Gamma_2=\{\alpha\in \mathcal{M}: l(\alpha)\leq 2\}$, $I_{\alpha,n}:=I_{\alpha}[1]_{t_n,t_{n+1}}$
are the multiple It\^o integrals for the index $\alpha$ over the time interval $[t_n,t_{n+1}]$,
and $\phi^n=\big(\phi_1^n,\ldots,\phi_d^n\big)^\top$ with its $i$-th component
\begin{equation*}
\phi_i^n =\; b_i\Delta +\sum_{j_1=1}^d\sigma_{ij_1}I_{(j_1),n}+\sum_{j_1,j_2=1}^d L^{j_1}\sigma_{ij_2} I_{(j_1,j_2), n}
+\frac{1}{2}\sum_{j_1=1}^d\big(L^{j_1} b_i+L^0\sigma_{ij_1}\big)\Delta\, I_{(j_1), n}+\frac{1}{2}L^0b_i \Delta^2
\end{equation*}
with
$I_{(j_1,j_2), n}=\int_{t_n}^{t_{n+1}}\int_{t_n}^{s_2}dW_{s_1}^{j_1}dW_{s_2}^{j_2}$  for $ j_1, j_2 \in \{1,\ldots,d\}$.
In the sequel, if there is no confusion, for a function $a=a(t,x)$,
we denote $a(t_n,X^n)$ by $a$.

\subsection{The Malliavin calculus on SDE and BSDE}
Suppose that $H$ is a real separable Hilbert space with scalar product denoted
by $\langle\cdot, \cdot\rangle_H$. The norm of an element $h \in H$ will be denoted by $\|h\|_H$.
Let $\mathcal W = \{W(h),h \in H \}$ denote an isonormal Gaussian process associated
with the Hilbert space $H$ on $(\Omega,\mathcal{F}, \mathbb F, P)$.

For the Brownian motion $W_t=(W_t^1,\ldots,W_t^d)^\top$, we define a random
variable of the form $F=f\big(W(h^1),\ldots,W(h^d)\big)$, where $h^k=\big(h^{k,1},\ldots,h^{k,d}\big)$
and
\begin{equation}\label{11}
W(h^{k})=\int_0^\infty h_t^{k}d W_t,
\qquad 1\leq k\leq d.
\end{equation}
 Similarly we define  $W^i(h^{k,i})=\int_0^\infty h_t^{k,i}d W_t^i$, then it is easy to see $W(h^{k})=\sum\limits_{i=1}^d W^i(h^{k,i}).$
 For the index $\alpha=(i)$ $(0\leq i\leq d)$, let $D_{t}^{\alpha}=D_{i,t}$  ($0\leq t \leq T$) be the Malliavin derivative of order one w.r.t. $W_s^i,$ with the convention that $D_{0,t}$ is just the identity,
i.e.,  $D_{0,t}=1$ and
\begin{equation*}
D_{i,t} F=\sum\limits_{k=1}^d\frac{\partial f}{\partial x^{k,i}}(W(h^1),\ldots, W(h^d))h_t^{k,i}, \quad \textrm{for} \quad i=1,\ldots,d.
\end{equation*}
From an intuitive point of view $D_{i,t} F$ represents the derivative of $F$
w.r.t. the increment of $i$-th Brownian motion $W^i$ corresponding to $t$.
 We will sometimes use the following intuitive notation
\begin{equation*}
D_{0,t}=1, \quad  D_{i,t} F=\frac{\partial F}{\partial \Delta_t^i},
\end{equation*}
where $\Delta_t^i=\Delta W_n^i=W_{t_{n+1}}^i-W_{t_n}^i$ for $1\leq i\leq d$ and $t_n\leq t\leq t_{n+1}$.
For a random variable $X=(X^1,\ldots,X^d)^\top$, we assume
\begin{equation}
D_{i,t}X=(D_{i,t} X^1,\ldots,D_{i,t} X^d)^\top, \quad
D_{t}X=\big(D_{1,t} X,\ldots,D_{d,t} X\big)_{d\times d}
\end{equation}
and
 \begin{equation}
D_{s_1\cdots s_l}^{\alpha}=D_{s_1\cdots s_l}^{(j_1,\ldots,j_l)}=D_{j_1, s_1}, \cdots D_{j_l, s_l}
 \end{equation}
 for the multi-indices $\alpha=(j_1,\ldots,j_l)\in \mathcal{A}_l$ with $j_i\in \{0,1,\cdots, d\}$ ($i=1,\ldots,l$),
where $\mathcal{A}_l=\{\alpha\in \mathcal{M}: l(\alpha)= l\}$,
and $s_i\in [0,T]$.
For any integer $p \ge 1$, $\mathbb{D}^{k,p}$ is the domain of $\mathbb{D}^k$
($k\in \mathbb{N}$) in $L^p(\Omega)$, that is,
$\mathbb{D}^{k,p}$ is the closure of the class of smooth random variables $F$ w.r.t.
the norm
$$\|F\|_{k,p}^p = \mathbb{E}[|F|^p]+\sum\limits_{j=1}^k \sum\limits_{|\alpha|=l}\int_0^T\cdots\int_0^T\mathbb{E}[|D^\alpha_{s_1,\ldots,s_l} F|^{p}]d s_1\ldots d s_l.$$
 For $p = 2$, the space $\mathbb{D}^{1,2}$ is a Hilbert space with the scalar product
$$\langle F,G\rangle= \mathbb{E}[FG]+\mathbb{E}[\langle DF, DG \rangle_H],$$
where $\ds \langle DF,DG\rangle_H: =\int_0^T \sum\limits_{i=1}^d D_{i,t} F \, D_{i,t} G dt$.

For $t_n<r\leq s\leq t\leq t_{n+1}$ and $i,j,k \in \{1,2,\ldots,d\}$,
by taking the Malliavin derivative $D_{j,t}$, $D_{i, s}D_{j, t}$ and $D_{k,r}D_{i,s}D_{j, t}$
to the multiple integral $I_{(j_1,j_2), n}$, we easily get
\begin{equation*}
\begin{aligned}
D_{j,t}I_{(j_1,j_2), n}
=\,&\Big(\int_{t_n}^tdW_{s_1}^{j_1}\Big)\delta_{jj_2}+\int_t^{t_{n+1}}D_{j,t}\big(W_{s_2}^{j_1}-W_t^{j_1}+W_t^{j_1}-W_{t_n}^{j_1}\big)dW_{s_2}^j\\
=\,&(W_t^{j_1}-W_{t_n}^{j_1})\delta_{jj_2}+(W_{t_{n+1}}^{j_2}-W_t^{j_2})\delta_{jj_1},\\
D_{i, s}D_{j, t}I_{(j_1,j_2), n}
=\,&D_{i,s}(W_t^{j_1}-W_{t_n}^{j_1})\delta_{jj_2}+D_{i,s}\big(W_{t_{n+1}}^{j_2}-W_t^{j_2}\big)\delta_{jj_1} =\delta_{ij_1}\delta_{jj_2}, \, \\
D_{k,r}D_{i,s}D_{j, t}I_{(j_1,j_2), n} =&\,0.
\end{aligned}
\end{equation*}
Then for $s_1<s_2<s_3$ and $\alpha=(j_1,j_2.j_3)$, it holds that
\begin{equation}\label{phin}
\begin{array}{rl}
D_{s_1s_2s_3}^\alpha\phi_i^n= &\!\!\!\!
D_{j_3, s_3} \phi_i^n =\sigma_{ij_3}
+\sum\limits_{j=1}^d L^{j}\sigma_{ij_3}\big(W_{s_3}^{j}-W_{t_n}^j\big)
+\sum\limits_{j=1}^d L^{j_3}\sigma_{ij}\big(W_{t_{n+1}}^{j}-W_{s_3}^{j}\big)\\
& \hspace{1.5cm}
+\frac{1}{2}\big(L^{j_3} b_i+L^0\sigma_{ij_3}\big)\Delta, \quad j_1=j_2=0,\;\; j_3\geq 1,\\
D_{s_1s_2s_3}^\alpha\phi_i^n=
&\!\!\!\!
D_{s_2, j_2} D_{s_3, j_3}\phi_i^n =\sum\limits_{j=1}^d L^{j}\sigma_{ij_3}\delta_{jj_2}=L^{j_2}\sigma_{ij_3}, \quad j_1=0, \;\; j_2, j_3\geq 1,\\
D_{s_1s_2s_3}^\alpha \phi_i^n=
&\!\!\!\!
D_{j_1, s_1 } D_{j_2, s_2 }D_{j_3 , s_3 } \phi_i^n=0, \quad j_1, j_2, j_3\geq 1.
\end{array}
\end{equation}
For the Malliavin derivative operator $D_t$, we introduce the following two lemmas.
\begin{lem} [Integration-by-parts formula] \label{nualart}
For $F\in \mathbb{D}^{1,2}$, $u\in L^2(\Omega;H)$ and $i,j\in \{1,2,\ldots,d\}$,
we have
\begin{equation}\label{Malp}\begin{aligned}
&\mathbb{E}[F\int^T_0u_t dW_t^i]=\mathbb{E}[\int_0^TD_{i,t} F\, u_t dt], \quad
  D_{i, t}\int_0^Tu_s ds=\int_t^TD_{i, t}u_s ds, \\
&D_{i,t}\int_0^Tu_s d W_s^j=u_t\delta_{ij}+\int_t^TD_{i,t}u_s d W_s^j, \quad 0<t<T,
\end{aligned}\end{equation}
where $\delta_{ij}$ is the Kronecker delta function \cite{Nual95}.
\end{lem}
\begin{lem}[Chain-rule]\label{pro1}
If $\psi: \mathbb{R}^m\rightarrow \mathbb{R}$ is a continuously differential function
with bounded partial derivatives, $F=(F_1,\ldots, F_m)$ $(F_i\in \mathbb{D}^{1,2})$ is a random vector,
and the solution $\big(X_r^{t,x},Y_r^{t,x},Z_r^{t,x}\big)$ of \eqref{DFBSDEs2} is in $\mathbb{D}^{1,2}$, then
the following identities hold \cite{KPQ97,Nual95}
\begin{equation}\label{DY}\begin{aligned}&
D_{i,t}\psi(F)
=\sum\limits_{j=1}^d\frac{\partial \psi}{\partial x_j}(F)D_{i,t}F_j, \\&
D_sX_r^{t,x}I_{s\leq r}
=\nabla_xX_r^{t,x}(\nabla_x X_s^{t,x})^{-1}\sigma(s,X_s^{t,x})I_{s\leq r},\quad
D_sY_r^{t,x}I_{s\leq r}
= \nabla_x Y_r^{t,x}D_sX_r^{t,x}I_{s\leq r}.
\end{aligned}\end{equation}
\end{lem}

\section{The C-N Scheme for solving decoupled FBSDEs}
 \qquad\\
Let $(X_t^{t_n,X^n}, Y_t^{t_n,X^n}, Z_t^{t_n,X^n})_{t_n\leq t\leq T}$ ($0\leq n\leq N-1$) be the
solution of the FBSDEs \eqref{DFBSDEs2} with $t$ and $x$ replaced by $t_n$ and $X^n$,
respectively. Then we have
\begin{equation}\label{s3:e1}
Y_{t_n}^{t_n,X^n}=\,Y_{t_{n+1}}^{t_n,X^n}+\int_{t_n}^{t_{n+1}}f(s,X_s^{t_n,X^n},Y_s^{t_n,X^n},Z_s^{t_n,X^n})
\,ds-\int_{t_n}^{t_{n+1}}Z_s^{t_n,X^n}\,dW_s.
\end{equation}
By taking the conditional mathematical expectation $\mathbb{E}_{t_n}^{X^n}[\cdot]$ to the
above equation gives
\begin{equation}
\label{s3:e2}
Y_{t_n}^{t_n,X^n}=\;\mathbb{E}_{t_n}^{X^n}[Y_{t_{n+1}}^{t_n,X^n}]+\int_{t_n}^{t_{n+1}}\mathbb{E}_{t_n}^{X^n}[f_s^{t_n,X^n}]\,ds
\end{equation}
with $f_s^{t_n,X^n}:=f(s,X_s^{t_n,X^n},Y_s^{t_n,X^n},Z_s^{t_n,X^n})$.
When $n=N-1$, we use the Euler method to approximate the integral in (\ref{s3:e2}) and obtain
\begin{equation}\label{s2.3}
\begin{aligned}
Y_{t_{N-1}}^{t_{N-1},X^{N-1}}
&=\,\mathbb{E}_{t_{N-1}}^{X^{N-1}}\big[\varphi(X^N)\big]+\Delta^2 f_{t_{N-1}}^{t_{N-1},X^{N-1}}
+\sum\limits_{j=1}^2R_{yj}^{N-1},
\end{aligned}
\end{equation}
where
\begin{equation*}
R_{y1}^{N-1}= \int_{t_{N-1}}^{t_{N}}\big\{\mathbb{E}_{t_{N-1}}^{X^{N-1}}[f_s^{t_{N-1},X^{N-1}}]
 -f_{t_{N-1}}^{t_{N-1},X^{N-1}}\big\}ds, \quad
R_{y2}^{N-1}= \mathbb{E}_{t_{N-1}}^{X^{N-1}}[\varphi(X_{t_{N}}^{t_{N-1},X^{N-1}})-\varphi(X^{N})].
\end{equation*}
For $0\leq n\leq N-2$, by using the trapezoidal rule
to approximate the integral in (\ref{s3:e2}), we deduce
\begin{equation}\label{s3:e4}
\begin{aligned}
Y_{t_n}^{t_n,X^n}
=&\;\mathbb{E}_{t_n}^{X^n}[Y_{t_{n+1}}^{t_{n+1},X^{n+1}}]+\frac{1}{2}\Delta f_{t_n}^{t_n,X^n}
+\frac{1}{2}\Delta \mathbb{E}_{t_n}^{X^n}[f_{t_{n+1}}^{t_{n+1},X^{n+1}}]
+\sum\limits_{j=1}^2R_{yj}^n,
\end{aligned}
\end{equation}
where
\begin{equation}\label{ry1}
\begin{array}{rl}
&\;R_{y1}^n=\int_{t_n}^{t_{n+1}}\{\mathbb{E}_{t_n}^{X^n}[f_s^{t_n,X^n}]
 -\frac{1}{2}\mathbb{E}_{t_n}^{X^n}[f_{t_{n+1}}^{t_n,X^n}]
 -\frac{1}{2}f_{t_n}^{t_n,X^n}\}ds,\\
&\;R_{y2}^n= \mathbb{E}_{t_n}^{X^n}[Y_{t_{n+1}}^{t_n,X^n}-Y_{t_{n+1}}^{t_{n+1},X^{n+1}}]+\frac{1}{2}\Delta \mathbb{E}_{t_n}^{X^n}[f_{t_{n+1}}^{t_n,X^n}-f_{t_{n+1}}^{t_{n+1},X^{n+1}}].
\end{array}
\end{equation}

Let $\Delta W_n=W_{t_{n+1}}-W_{t_n}$ for $0\leq n\leq N-1$.
 Multiplying (\ref{s3:e1}) by $\Delta W_{n}^{\top}$, taking the conditional mathematical
expectation $\mathbb{E}_{t_n}^{X^n}[\cdot]$ on both sides of the
derived equation,  and then using the It\^o isometry formula we obtain
\begin{equation*}\label{s3:e5}
-\mathbb{E}_{t_n}^{X^n}[Y_{t_{n+1}}^{t_n,X^n}\Delta W_{n}^{\top}]=
\int_{t_n}^{t_{n+1}}\mathbb{E}_{t_n}^{X^n}[f_s^{t_n,X^n}\Delta W_n^{\top}]\,ds
-\int_{t_n}^{t_{n+1}}\mathbb{E}_{t_n}^{X^n}[Z_s^{t_n,X^n}]\,ds.
\end{equation*}
When $n = N-1$, the Euler scheme is applied to approximate the integral in the above
equation, then
\begin{equation}\label{ZN-1}
Z_{t_{N-1}}^{t_{N-1},X^N-1} = \frac{1}{\Delta ^2}
\mathbb{E}_{t_{N-1}}^{X^{N-1}}\big[\varphi(X^N)\Delta W_{N-1}^\top\big]
+ \frac{1}{\Delta^2}\sum\limits_{j=1}^2R_{zj}^{N-1},
\end{equation}
where
\begin{equation*}\label{rz}
\begin{array}{rl}
R_{z1}^{N-1}=&\!\!\!\! \int_{t_{N-1}}^{t_{N}}\mathbb{E}_{t_{N-1}}^{X^{N-1}}
[f_s^{t_{N-1},X^{N-1}}\Delta W_{N-1}^\top]ds
-\int_{t_{N-1}}^{t_{N}}\mathbb{E}_{t_{N-1}}^{X^{N-1}}\big([Z_s^{t_{N-1},X^{N-1}}]
-Z_{t_{N-1}}^{t_{N-1},X^{N-1}}\big)ds, \\
R_{z2}^{N-1}=&\!\!\!\! \mathbb{E}_{t_{N-1}}^{X^{N-1}}[\varphi(X_{t_{N}}^{t_{N-1},X^{N-1}})\Delta W_{N-1}^\top
-\varphi(X^{N})\Delta W_{N-1}^\top].
\end{array}
\end{equation*}
For $0\leq n\leq N-2$, following similar derivation of the equation (\ref{s3:e4}), we obtain the second
reference equation as
\begin{equation}\label{s2e13}
\begin{aligned}
\frac{1}{2}\Delta Z_{t_n}^{t_n,X^n}
=-\frac{1}{2}\Delta
\mathbb{E}_{t_n}^{X^n}[Z_{t_{n+1}}^{t_{n+1},X^{n+1}}]+\mathbb{E}_{t_n}^{X^n}[Y_{t_{n+1}}^{t_{n+1},X^{n+1}}\Delta W_{n}^{\top}]
+\frac{1}{2}\Delta \mathbb{E}_{t_n}^{X^n}[f_{t_{n+1}}^{t_{n+1},X^{n+1}}\Delta W_{n}^{\top}]
+\sum\limits_{j=1}^2R_{zj}^n,
\end{aligned}\end{equation}
where
\begin{equation}\label{crz}
\begin{array}{rl}
R_{z1}^n=
&\!\!\!\!
\int_{t_n}^{t_{n+1}}\mathbb{E}_{t_n}^{X^n}[f_s^{t_n,X^n}\Delta
W_n^{\top}]ds-\frac{1}{2}\Delta
\mathbb{E}_{t_n}^{X^n}[f_{t_{n+1}}^{t_n,X^n}\Delta  W_{n}^{\top}]\\
&\!\!\!\!
 -\int_{t_n}^{t_{n+1}}\{\mathbb{E}_{t_n}^{X^n}[Z_s^{t_n,X^n}]
-\frac{1}{2}\mathbb{E}_{t_n}^{X^n}[Z_{t_{n+1}}^{t_n,X^n}]-\frac{1}{2}Z_{t_n}^{t_n,X^n}\}\,ds, \\
R_{z2}^n=
&\!\!\!\!
-\frac{1}{2}\Delta \mathbb{E}_{t_n}^{X^n}[Z_{t_{n+1}}^{t_n,X^n}-Z_{t_{n+1}}^{t_{n+1},X^{n+1}}]
+\mathbb{E}_{t_n}^{X^n}[\big(Y_{t_{n+1}}^{t_n,X^n}-Y_{t_{n+1}}^{t_{n+1},X^{n+1}}\big)\Delta W_{n}^{\top}]\\
&\!\!\!\!
+\frac{1}{2}\Delta \mathbb{E}_{t_n}^{X^n}[\big(
f_{t_{n+1}}^{t_n,X^n}-f_{t_{n+1}}^{t_{n+1},X^{n+1}}\big)\Delta W_{n}^{\top}].
\end{array}
\end{equation}

Let $(Y^n,Z^n)$ denote the approximation to the exact solution $(Y_{t_n}^{t_n,X^n},Z_{t_n}^{t_n,X^n})$ of BSDE \eqref{s3:e1} for $n=N-1,\ldots,0$.
For simple representation, we denote $f^n:=f(t_n,X^n,Y^n,Z^n)$.
Now, based on the reference equations  \eqref{s2.3}, \eqref{s3:e4}, \eqref{ZN-1} and \eqref{s2e13},
 we introduce the Crank-Nicolson scheme (Scheme 2.1 proposed in \cite{ZLF14}) for solving decoupled FBSDEs \eqref{DFBSDEs1}.
\begin{sch}\label{sch1}
Suppose that the initial condition $X_0$ for the forward SDE in \eqref{DFBSDEs1} and
the terminal condition $\varphi$ for the BSDE in \eqref{DFBSDEs1} are given.
\begin{enumerate}
\item For $n=N-1$, $t_N-t_{N-1}=\Delta^2$, solve $X^N$, $Y^{N-1}$ and $Z^{N-1}$ by
\begin{subequations}\label{scheme-step1}
\begin{align}
X^N &= \,X^{N-1} + b(t_{N-1},X^{N-1})\Delta^2 + \sigma(t_{N-1},X^{N-1})\Delta W_{N-1}^\top,\\
Z^{N-1}&=\,\frac{1}{\Delta^2}\mathbb{E}_{t_{N-1}}^{X^{N-1}}[Y^N\Delta W_{N-1}^\top], \label{ZN-11}\\
Y^{N-1}&=\,\mathbb{E}_{t_{N-1}}^{X^{N-1}}[Y^N]+\Delta^2 f^{N-1}.
\end{align}
\end{subequations}
\item
For $n = N-2,\dots,0$, solve $X^{n+1}$, $Y^n$, $Z^n$ by
\begin{subequations}\label{scheme-n}
\begin{align}
X^{n+1}&=X^n+\sum\limits_{\alpha\in \Gamma_2 \backslash \{v\}}g_\alpha(t_n, X^n)I_{\alpha, n},\\
\label{sch3:e1}
\frac{1}{2}\Delta Z^n&= -\frac{1}{2}\Delta\mathbb{E}_{t_n}^{X^n}[Z^{n+1}]+\mathbb{E}_{t_n}^{X^n}[Y^{n+1}\Delta
W_{n}^{\top}]+\frac{1}{2}\Delta\mathbb{E}_{t_n}^{X^n}[f^{n+1}\Delta W_{n}^{\top}],\\
\label{sch3:e2}
Y^n&= \mathbb{E}_{t_n}^{X^n}[Y^{n+1}]+\frac{1}{2}\Delta f^n
+\frac{1}{2}\Delta \mathbb{E}_{t_n}^{X^n}[f^{n+1}].
\end{align}
\end{subequations}
\end{enumerate}
\end{sch}

\begin{rem}
%
\begin{enumerate}
\item
In 2006, the authors in \cite{ZCP06} proposed the following
scheme for solving BSDE.
\begin{equation}\label{theta-scheme}
\begin{aligned}
Y^n=\,& \mathbb{E}_{t_n}^x[Y^{n+1}] + \Delta t_n[(1-\theta_1^n)\mathbb{E}_{t_n}^x
[f^{n+1}]+\theta_1^n f^n],\\
0=\,& \mathbb{E}_{t_n}^x[Y^{n+1}\Delta W_{n}^\top] + \Delta t_n(1-\theta_2^n)\mathbb{E}_{t_n}^x
[f^{n+1}\Delta W_{n}^\top]
-\Delta t_n\{(1-\theta_2^n)\mathbb{E}_{t_n}^x[Z^{n+1}]+\theta_2^n Z^n\},
\end{aligned}
\end{equation}
where $\Delta t_n = t_{n+1}-t_n$ and $f^n=f(t_n,Y^n,Z^n)$ for $n = N-1,\ldots,0$ with the parameters $\theta_1^n$ and
$\theta_2^n$ in $[0,1]$.
The C-N scheme for BSDE is the case $\theta_1^n = \theta_2^n = \frac{1}{2}$.
The second-order error estimate results were proved in \cite{ZLJ13}.
\item
In 2012, the authors in \cite{ZLZ12} developed the following $\theta$-scheme for BSDE:
Given the terminal values $Y^N$ and $Z^N$, solve $Y^n$ and $Z^n$ by
\begin{equation}\label{G-theta-scheme}
\begin{aligned}
Y^n =\,& \mathbb{E}_{t_n}^x[Y^{n+1}]
+ \Delta t_n\big[(1-\theta_1) \mathbb{E}_{t_n}^x[f^{n+1}]+\theta_1 f^n\big],\\
\theta_3\Delta t_n Z^n =\,& \theta_4 \Delta t_n \mathbb{E}_{t_n}^x[Z^{n+1}] + (\theta_3 - \theta_4)
\mathbb{E}_{t_n}^x[Y^{n+1}\Delta W_{n}^\top]
+(1-\theta_2)\Delta t_n \mathbb{E}_{t_n}^x[f^{n+1}\Delta W_{n}^\top],
\end{aligned}
\end{equation}
where $\Delta t_n = t_{n+1}-t_n$ and $f^n=f(t_n,Y^n,Z^n)$ for $n = N-1,\ldots,0$ with the deterministic parameters
$\theta_i \in [0,1]$ $(i = 1,2)$,
$\theta_3\in (0,1]$, and $\theta_4\in [-1,1]$ constrained by
$|\theta_4| \leq \theta_3$.

When $\theta_i = \frac{1}{2} (i=1,2,3)$ and $\theta_{4}=-\frac{1}{2}$,
the above scheme becomes the C-N scheme for BSDE.
In \cite{ZLZ12} the second-order convergence rate of the above scheme
was theoretically proved with the parameters $\theta_i \in [0,1]$ $(i = 1,2)$,
$\theta_3\in (0,1]$, and $\theta_4\in [-1,1]$ constrained by
$|\theta_4| < \theta_3$.
\item
By introducing the Gaussian process
$$
\Delta \tilde {W}_n = 4\frac{W_{t_{n+1}}-W_{t_n}}{\Delta t_n}
- 6\frac{\int_{t_n}^{t_{n+1}}(s-t_n)dW_s}{(\Delta t_n)^2},
$$
the authors in \cite{CM14} proposed the following scheme for solving BSDE.
\begin{equation}\label{scheme-2}
Z^{n} = \mathbb{E}_{t_n}[\big(Y^{n+1}+\Delta t_n f^{n+1}\big)\Delta\tilde{W}_n],\quad
Y_{n}=\mathbb{E}_{t_n}[Y_{n+1}]+\frac{\Delta t_n}{2}\big(f^n+\mathbb{E}_{t_n}[f^{n+1}]\big),
\end{equation}
where $f^n= f(X^n,Y^n,Z^n)$.
The authors in \cite{CM14} only obtained the second-order
convergence rate of the above scheme
for $X^{n}=X_{t_n}$, i.e., the forward SDE was not discretized.
Note that the introduced stochastic process $\Delta\tilde W_n$ in the scheme
will cause computation expensive for solving BSDE,
and further it is too complex to use the scheme to solve FBSDEs.
\end{enumerate}
\end{rem}

\section{Error estimates of the C-N scheme}\label{3}

\subsection{Assumptions on approximations of $X_t$}
It is obvious that the accuracy of Scheme \ref{sch1} depends on the accuracy of \eqref{Xnn10} for
solving the forward SDE $X_t$ in \eqref{DFBSDEs1}.
In this subsection, to investigate the effect of approximation of forward SDE on the approximation solutions
$(Y^n,Z^n)$ in the Crank-Nicolson scheme, the following assumptions are made.

\begin{assum}\label{hyp0}
Suppose that $X_0$ is $\mathcal{F}_0$-measurable with $\mathbb{E}[|X_0|^{2}]<\infty$,
and that $b$ and $\sigma$ are $L^2$-measurable in $(t,x)\in [0, T]\times \mathbb{R}^d$,
are linear growth bounded and uniformly Lipschitz continuous, i.e.,
there exist positive constants $K$ and $L$ such that
\begin{equation} \label{linear-Lip}
\begin{aligned}
&|b(t,x)|^2 \leq K(1+|x|^2), \quad\qquad |\sigma(t,x)|^2 \leq K(1+|x|^2),\\
&|b(t,x)-b(t,y)| \leq L|x-y|,   \quad |\sigma(t,x)-\sigma(t,y)| \leq L|x-y|.
\end{aligned}
\end{equation}
\end{assum}

\begin{assum}\label{hyp1}
There is a constant $K'>0$ such that the coefficient matrix $\sigma$ satisfies
the uniformly elliptic condition
\begin{equation}
\sigma(t,x)\sigma(t,x)^\top\geq \frac{1}{K'}I_{d\times d}, \quad \forall \,(t,x)\in [0,T]\times \mathbb{R}^d.
\end{equation}
\end{assum}

Under the Assumption \ref{hyp0}, if $\mathbb{E}[|X_0|^{2m}]<\infty$ for some integer
$m\geq 1$, the solution of SDE in \eqref{DFBSDEs2} has the estimate
\begin{equation}\label{est_Xs}
\mathbb{E}\big[|X_s^{t_n,X^n}|^{2m}\big] \leq (1+|X^n|^{2m})e^{C(s-t_n)},
\end{equation}
for any $s\in[t_n,T]$,
where $C$ is a positive constant depending only on the constants $K$, $L$ and $m$ \cite{KP92}.

In fact, the weak order-2 It\^{o}-Taylor schemes \eqref{Xnn10} for solving SDE
in \eqref{DFBSDEs1}
have the following approximation properties and the stability property
(see Proposition 5.11.1 in \cite{KP92} and Assumption 4.2 in\cite{ZZJ14}).
\begin{assum}\label{hyp3}
The approximation solution $X^{n}$ ($n=0,1,\ldots,N-1$) has the approximation properties
\begin{subequations}\label{Eq45}
\begin{align}\label{c5Ex}
& \big|\mathbb{E}_{t_n}^{X^n}\big[g(X^{n+1})-g(X^n)\big]\big|
\leq C(1+|X^n|^{2r_1})\Delta,\\ \label{weak-beta}
&\big|\mathbb{E}_{t_n}^{X^n}\big[g(X_{t_{n+1}}^{t_n,X^n})-g(X^{n+1})\big]\big|
\leq C(1+|X^n|^{2r_2})\Delta^{3},\\ \label{c5gw}
&\big|\mathbb{E}_{t_n}^{X^n}[\big(g(X_{t_{n+1}}^{t_n,X^n})-g(X^{n+1})\big)\Delta W_{n}^{\top}]\big|
\leq  C(1+|X^n|^{2r_3})\Delta^{3},
\end{align}
\end{subequations}
and the stable estimate property
\begin{equation}\label{c5st}
\max\limits_{0\le n\le N}\mathbb{E}\big[|X^n|^r\big]\leq C(1+\mathbb{E}[|X_0|^r]),
\end{equation}
where $r_i$ $(i = 1,2,3)$ and $r$ are positive integers,
and $C$ is a positive constant depending on $g\in C_b^{2\beta+2}$.
\end{assum}

The proposition below for the approximation \eqref{Xnn10} of $X_t$ holds as well.
\begin{pro}\label{hyp2}
Suppose the functions $b,\sigma\in C_{b}^{1,2}$.
For $n=0,1,\ldots, N-1$, we have the estimates
\begin{align}\label{phi}
\sum\limits_{\alpha\in \Gamma_2\backslash \{v\}}\mathbb{E}_{t_n}^{X^n}[|g_\alpha(t_{n+1},X^{n+1})I_{\alpha, n+1}|^2]\leq&\; C\Delta ,\\
\label{phi1}
\mathbb{E}_{t_n}^{X^n}\big[|\sigma(t_{n+1},X^{n+1})-\sigma(t_n,X^n)|^2\big]\leq&\; C\Delta ,
\end{align}
where $C$ is a positive constant independent of $X^{n}$, $X^{n+1}$, and the time partition.
\end{pro}

\subsection{Error equations}
Let $(X_r^{t,x}, Y_r^{t,x},Z_r^{t,x})_{t\leq r\leq T}$  be the solution of the FBSDEs (\ref{DFBSDEs2})
with the terminal condition $Y_T^{t,x}=\varphi(X_T^{t,x})$,
$(X^{n+1},Y^n,Z^n)$ ($0\leq n\leq N-1$) be its
approximation solution of Scheme \ref{sch1}, and
let the truncation errors $R_{y1}^{n}$, $R_{y2}^{n}$, $R_{z1}^{n}$ and $R_{z2}^{n}$
be defined in \eqref{ry1} and \eqref{crz} for $0\leq n \leq N$, respectively.

For the sake of presentation simplicity, we denote
\begin{align*}
&e_Y^n:=Y_{t_n}^{t_n,X^n}-Y^n, \quad e_Z^n:=Z_{t_n}^{t_n,X^n}-Z^n, \\
&e_f^n=f(t_n,X^n,Y_{t_n}^{t_n,X^n},Z_{t_n}^{t_n,X^n})-f(t_n,X^n,Y^n,Z^n),
\end{align*}
for $n =N-2,\ldots,1,0$.
Subtracting \eqref{s3:e4} and \eqref{s2e13} from \eqref{sch3:e2} and \eqref{sch3:e1}, respectively, we get
\begin{align}
&e_Y^n= \mathbb{E}_{t_n}^{X^n}[e_Y^{n+1}]
+\frac{1}{2}\Delta  e_f^n+\frac{1}{2}\Delta \mathbb{E}_{t_n}^{X^n}[e_f^{n+1}]+\sum\limits_{j=1}^2R_{yj}^n, \label{a1} \\
&\Delta e_Z^n=-\Delta \mathbb{E}_{t_n}^{X^n}[e_Z^{n+1}]+2\mathbb{E}_{t_n}^{X^n}[e_Y^{n+1}\Delta W_{n}^\top]
+\Delta \mathbb{E}_{t_n}^{X^n}[e_f^{n+1}\Delta W_{n}^\top]+2\sum\limits_{j=1}^2R_{zj}^n. \label{s3:e32z}
\end{align}
Let $e_{\nabla Y}^n := \nabla_{x^n} Y_{t_n}^{t_n,X^n}-\nabla_{x^n} Y^n$
and $e_{\nabla Z}^n := \nabla_{x^n} Z_{t_n}^{t_n,X^n}-\nabla_{x^n} Z^n$.
Taking variations w.r.t. $X^{n}$
on both sides of equations \eqref{a1} and \eqref{s3:e32z} gives us the following
two equations:
\begin{equation}\label{s3:e218}
e_{\nabla Y}^n=\; \mathbb{E}_{t_n}^{X^n}[e_{\nabla Y}^{n+1}\nabla_{x^n} X^{n+1}]
+\frac{1}{2}\Delta  e_{\nabla f}^n+\frac{1}{2}\Delta \mathbb{E}_{t_n}^{X^n}[e_{\nabla f}^{n+1}\nabla_{x^n} X^{n+1}]+\sum\limits_{j=1}^2\nabla_{x^n}R_{yj}^n
\end{equation}
and
\begin{equation}\label{s3:e3221}
\begin{array}{ll}
 \Delta e_{\nabla Z}^n=&\!\!\!\!-\Delta \mathbb{E}_{t_n}^{X^n}[e_{\nabla Z}^{n+1}\nabla_{x^n} X^{n+1}]+2\mathbb{E}_{t_n}^{X^n}[\Delta W_{n} e_{\nabla Y}^{n+1}\nabla_{x^n} X^{n+1}]\\
&+\Delta \mathbb{E}_{t_n}^{X^n}[\Delta W_{n}e_{\nabla f}^{n+1} \nabla_{x^n} X^{n+1}]+2\sum\limits_{j=1}^2\nabla_{x^n}R_{zj}^n,
\end{array}
\end{equation}
where $\Delta W_{n}=(\Delta W_{n}^1, \ldots, \Delta W_{n}^d)^\top $,
$e_{\nabla Y}^{n+2}:=\big(e_{\nabla Y}^{1,n+2},\ldots, e_{\nabla Y}^{d,n+2}\big)=\big(\nabla_{x_1^n} e_Y^{n+2},\ldots, \nabla_{x_d^n} e_Y^{n+2}\big)$, and
$$
\begin{aligned}
e^{n}_{\nabla f} =\;& f_X^{t_n,X^n}-f_X^n
+\big(f_Y^{t_n,X^n}-f_Y^n\big)\nabla_{x^n} Y_{t_{n}}^{t_n,X^n}+f_Y^n e_{\nabla Y}^n\\
&+\big(f_Z^{t_n,X^n}-f_Z^n\big)\nabla_{x^n} Z_{t_n}^{t_n,X^n}
+f_Z^n e_{\nabla Z}^{n}.
\end{aligned}
$$

%

\subsection{Main error estimate results}

Now we state our  main error estimate results in Theorems \ref{thm1} and \ref{thm2} below.

\begin{thm}\label{thm1}
For the weak order-2 It\^o-Taylor approximation 
$X^{n+1}$ satisfying \eqref{Xnn10},
if $b,\sigma\in C_b^{1,3}$ and $f\in C_b^{1,2,2,2}$,
then under Hypotheses \ref{hyp0}, 
for $0\leq n\leq N-2$, it holds that
\begin{equation}\label{eq:thm1}\begin{aligned}
&\mathbb{E}\big[|e_Y^n|^2+|e_Z^n|^2+|e_{\nabla Y}^n|^2+\Delta |e_{\nabla Z}^n|^2\big]\\
\leq \;&
C\,\mathbb{E}
\big[|e_Y^{N-1}|^2+|e_Z^{N-1}|^2
+|e_{\nabla Y}^{N-1}|^2+\Delta |e_{\nabla Z}^{N-1}|^2\big]\\&
+C\sum_{i=n}^{N-2}\sum\limits_{j=1}^2\mathbb{E}\Big[\frac{1}{\Delta^3}
\Big(|\mathbb{E}_{t_i}^{X^i}[R_{yj}^{i+1}\Delta W_{i}^\top]|^2+|R_{zj}^i-\mathbb{E}_{t_i}^{X^i}[R_{zj}^{i+1}]|^2\Big)\\&
\hspace{2cm}+\frac{1}{\Delta}
\Big(|\mathbb{E}_{t_i}^{X^i}[R_{yj}^{i+1}]|^2
+|\mathbb{E}_{t_i}^{X^i}[\nabla_{x^i} R_{yj}^{i+1}]|^2
+|\mathbb{E}_{t_i}^{X^i}[\nabla_{x^i} R_{zj}^{i+1}]|^2\\& \hspace{2.5cm}
+|\nabla_{x^i}\mathbb{E}_{t_i}^{X^i}[R_{yj}^{i+1}\Delta W_{i}^\top]|^2+|R_{yj}^i|^2+|\nabla_{x^i} R_{yj}^{i}|^2+|\nabla_{x^i} R_{zj}^{i}|^2\\& \hspace{2.5cm}
+\mathbb{E}_{t_i}^{X^i}[|R_{zj}^{i+1}|^2]
+\mathbb{E}_{t_i}^{X^i}[|\nabla_{x^{i+1}}R_{zj}^{i+1}|^2]
\Big)\Big],
\end{aligned}\end{equation}
where $C$ is a generic positive constant depending on $d$, $T$, $K'$, and upper
bounds of derivatives of $b$, $\sigma$ and $f$.
\end{thm}

\begin{thm}\label{thm2}
Suppose  $b,\sigma\in C_b^{3,6}$, $f\in C_b^{3,6,6,6}$, and
$\varphi\in C_b^{7+\alpha}$ for some $\alpha\in (0,1)$.
Then for the weak order-2 It\^o-Taylor approximation solution
$X^{n+1}$, $0\leq n\leq N-2$, under Assumptions \ref{hyp0}--\ref{hyp3},
it holds that
\begin{equation}\label{eq:thm2}
\max_{0\leq n\leq N}\mathbb{E}\left[|e_Y^n|^2+|e_Z^n|^2+|e_{\nabla Y}^n|^2
+\Delta |e_{\nabla Z}^n|^2\right]
\leq C\Delta^4,
\end{equation}
where $C$ is a generic positive constant depending on $d$, $T$, $K'$, $K$, $L$, the initial value
of $X_t$ in \eqref{DFBSDEs1}, and upper bounds of derivatives
of $b$, $\sigma$, $f$ and $\varphi$.
\end{thm}
\begin{rem}
Scheme \ref{sch1} is stable, which is implied by Theorem \ref{thm1}, and its solution
continuously depends on terminal condition.
That is, for any given positive number $\varepsilon$, there exits a positive
number $\delta$, for different terminal conditions
$(Y_1^N,Z_1^N)$ and $(Y_2^N,Z_2^N)$,
if $\mathbb{E}[|Y_1^N-Y_2^N|^2]<\delta$ and $\mathbb{E}[|Z_1^N-Z_2^N|^2]<\delta$,
then for $0\leq n\leq N-1$, we have
\begin{equation*}
\begin{aligned}
&\mathbb{E}\big[|Y_1^n-Y_2^n|^2 + |Z_1^n-Z_2^n|^2
\big]<\varepsilon.
\end{aligned}
\end{equation*}

\end{rem}

\section{Proofs of the main results}
In this section, we will give rigorous proofs of Theorems \ref{thm1} and \ref{thm2}.
In the sequel, we will use $Var^{n}(G)$
to denote the conditional variance of random variable $G$, i.e.,
$Var^n(G)=\mathbb{E}_{t_n}^{X^n}[|G|^2]-|\mathbb{E}_{t_n}^{X^n}[G]|^2.\quad$

Before giving the proof of Theorem \ref{thm1}, we introduce the following useful lemma.
\begin{lem}\label{th1}
For  the weak order-2 It\^o-Taylor approximation $X^{n+1}$ satisfying \eqref{Xnn10},
let $b,\sigma\in C_b^{1,3}$ and $f\in C_b^{1,2,2,2}$.
Then for $0\leq n\leq N-2$, it holds that
\begin{equation}\label{ceq}
\begin{aligned}
&\;\, \Delta \big(|e_Y^{n}|^2+|e_Z^{n}|^2
+|e_{\nabla Y}^{n}|^2+|e_{\nabla Z}^{n}|^2\big)\\
\leq & \;
C\Delta \mathbb{E}_{t_n}^{X^n}\big[|e_Y^{n+1}|^2+|e_Z^{n+1}|^2+|e_{\nabla Y}^{n+1}|^2
+|e_{\nabla Z}^{n+1}|^2\big]+CVar^{n}(e_Y^{n+1})+CVar^{n}(e_{\nabla Y}^{n+1})\\
&
+C\sum\limits_{j=1}^2\Big(\Delta|R_{yj}^{n}|^2+\Delta|\nabla_{x^n}R_{yj}^{n}|^2
+\frac{1}{\Delta}|R_{zj}^{n}|^2+\frac{1}{\Delta }|\nabla_{x^n}R_{zj}^n|^2\Big),
\end{aligned}
\end{equation}
where $C$ is a positive generic constant depending only on $d$, and upper
bounds of derivatives of $b$, $\sigma$, $f$ and $\varphi$.
\end{lem}
\begin{proof}
By \eqref{a1} and the Lipschitz continuity of function $f$, we easily deduce
\begin{equation*}\label{s3:e67}
|e_Y^{n}|\leq
(1+\frac{L'}{2}\Delta )\mathbb{E}_{t_{n}}^{X^{n}}[|e_Y^{n+1}|]
+\frac{L'}{2}\Delta\big(|e_Y^{n}|+|e_Z^{n}|
+\mathbb{E}_{t_{n}}^{X^{n}}[|e_Z^{n+1}|]\big)
+\sum\limits_{j=1}^2|R_{yj}^{n}|,
\end{equation*}
where $L'$ is the Lipschitz constant.
Then taking square on both sides of the above inequality and using the inequality
$\big(\sum\limits_{i=1}^ma_i\big)^2\leq m\sum\limits_{i=1}^ma_i^2$, we deduce
\begin{equation*}\label{s3:e671}\begin{aligned}
|e_Y^{n}|^2\leq &\;  C\mathbb{E}_{t_{n}}^{X^{n}}[|e_Y^{n+1}|^2]+C\Delta^2\big(|e_Y^{n}|^2+|e_Z^{n}|^2
+\mathbb{E}_{t_{n}}^{X^{n}}[|e_Z^{n+1}|^2]\big)
+C\sum\limits_{j=1}^2|R_{yj}^{n}|^2.
\end{aligned}\end{equation*}
Similarly, from the error equations \eqref{s3:e32z}--\eqref{s3:e3221} we obtain
\begin{equation*}
\begin{split}
\Delta |e_Z^{n}|^2 \leq\;&
C\Delta \mathbb{E}_{t_n}^{X^n}[|e_Z^{n+1}|^2]+CVar^{n}(e_Y^{n+1})\\&
+C\Delta^2\mathbb{E}_{t_n}^{X^n}[|e_Y^{n+1}|^2+|e_Z^{n+1}|^2]+\frac{C}{\Delta }\sum\limits_{j=1}^2|R_{zj}^{n}|^2,\\
\end{split}
\end{equation*}
\begin{equation*}
\begin{split}
|e_{\nabla Y}^n|^2
\leq \;& C\mathbb{E}_{t_n}^{X^n}[|e_{\nabla Y}^{n+1}|^2]
+C\Delta^2\big(|e_{\nabla Y}^{n}|^2+|e_{\nabla Z}^{n}|^2+|e_Y^{n}|^2+|e_Z^{n}|^2\big) \\&
+C\Delta^2\mathbb{E}_{t_n}^{X^n}[|e_{\nabla Y}^{n+1}|^2+|e_{\nabla Z}^{n+1}|^2+|e_Y^{n+1}|^2+|e_Z^{n+1}|^2]
+C\sum\limits_{j=1}^2|\nabla_{x^n} R_{yj}^n|^2,\\
\text{and}\qquad\qquad\quad\; & \\
\end{split}
\end{equation*}
\begin{equation*}
\begin{split}
\Delta |e_{\nabla Z}^n|^2
\leq \;& C\Delta \mathbb{E}_{t_n}^{X^n}[|e_{\nabla Z}^{n+1}|^2]+CVar^{n}(e_{\nabla Y}^{n+1})\\&
+C\Delta \mathbb{E}_{t_n}^{X^n}\big[|e_Y^{n+1}|^2+|e_Z^{n+1}|^2+|e_{\nabla Y}^{n+1}|^2+|e_{\nabla Z}^{n+1}|^2\big]
+\frac{C}{\Delta }\sum\limits_{j=1}^2|\nabla_{x^n} R_{zj}^n|^2.
\end{split}
\end{equation*}
Now combining the above four inequalities
yields
\begin{equation*}\label{ceq11}
\begin{aligned}
&\;\Delta \big(|e_Y^{n}|^2+|e_Z^{n}|^2+|e_{\nabla Y}^{n}|^2+|e_{\nabla Z}^{n}|^2\big)\\
\leq &\,
C\Delta \mathbb{E}_{t_n}^{X^n}\big[|e_Y^{n+1}|^2+|e_Z^{n+1}|^2+|e_{\nabla Y}^{n+1}|^2
+|e_{\nabla Z}^{n+1}|^2\big]+CVar^{n}(e_Y^{n+1})+CVar^{n}(e_{\nabla Y}^{n+1})\\
&+C\sum\limits_{j=1}^2\Big(\Delta |R_{yj}^{n}|^2
+\Delta|\nabla_{x^n}R_{yj}^{n}|^2+\frac{1}{\Delta }\left(|R_{zj}^{n}|^2
+|\nabla_{x^n}R_{zj}^n|^2\right)\Big).
\end{aligned}
\end{equation*}
The proof is completed.
\end{proof}

\subsection{Proof of Theorem \ref{thm1}}
Now, we give the proof of Theorem \ref{thm1}, which is divided into five steps.
In each step of the proof except the last step, we deduce an estimate for
$e_Y^n$, $e_Z^n$, $e_{\nabla Y}^n$, and $e_{\nabla Z}^n$ successively.

\begin{proof}
(1) {\bf The estimate of $e_{Y}^{n}$.}

By \eqref{a1}  we have
\begin{equation}\label{s3:e31b}
e_Y^{n+1}= \mathbb{E}_{t_{n+1}}^{X^{n+1}}[e_Y^{n+2}]
+\frac{1}{2}\Delta  \big(e_f^{n+1}+\mathbb{E}_{t_{n+1}}^{X^{n+1}}[e_f^{n+2}]\big)+\sum\limits_{j=1}^2R_{yj}^{n+1}.
\end{equation}
Inserting the $e_Y^{n+1}$ in \eqref{s3:e31b} into \eqref{a1} gives
\begin{equation}\label{eyn1}
\begin{aligned}
e_Y^n= &\mathbb{E}_{t_n}^{X^n}\big[\mathbb{E}_{t_{n+1}}^{X^{n+1}}[e_Y^{n+2}]\big]
+\frac{1}{2}\Delta  e_f^n+\Delta \mathbb{E}_{t_n}^{X^n}[e_f^{n+1}
+\frac{1}{2}e_f^{n+2}]+\sum\limits_{j=1}^2\big(R_{yj}^n
+\mathbb{E}_{t_n}^{X^n}[R_{yj}^{n+1}]\big).
\end{aligned}
\end{equation}
Then taking square on both sides of the above inequality
and using Young's inequality
$\ds (a+b)^2\leq (1+\gamma_1\Delta)a^2+(1+\frac{1}{\gamma_1\Delta })b^2$ (for any $\gamma_1>0$) yield
\begin{equation}\label{s3:e71}\begin{aligned}
|e_Y^n|^2 \leq &\; (1+\gamma_1\Delta)\mathbb{E}_{t_n}^{X^n}[|\mathbb{E}_{t_{n+1}}^{X^{n+1}}[e_Y^{n+2}]|^2]+\frac{C}{\gamma_1}(1+\gamma_1\Delta )\Bigg\{\Delta (|e_Y^n|^2+|e_Z^n|^2)\\
&\;+\Delta \mathbb{E}_{t_n}^{X^n}[|e_Y^{n+1}|^2+|e_Z^{n+1}|^2+|e_{Y}^{n+2}|^2+|e_{Z}^{n+2}|^2]
+\sum\limits_{j=1}^2\frac{|R_{yj}^n|^2+|\mathbb{E}_{t_n}^{X^n}[R_{yj}^{n+1}]|^2}{\Delta }\Bigg\}.
\end{aligned}\end{equation}

(2) {\bf The estimate of $e_{Z}^{n}$.}

Replacing the $n$ in  \eqref{s3:e32z} by $n+1$, we deduce
\begin{equation}\label{s3:eyn131}
- \mathbb{E}_{t_n}^{X^n}[e_Z^{n+1}]= \; \mathbb{E}_{t_{n}}^{X^{n}}[e_Z^{n+2}]
-\frac{2}{\Delta }\mathbb{E}_{t_{n}}^{X^{n}}[e_Y^{n+2}\Delta W_{n+1}^\top]
 -\mathbb{E}_{t_{n}}^{X^{n}}[e_f^{n+2}\Delta W_{n+1}^\top]
-\frac{2}{\Delta }\sum\limits_{j=1}^2\mathbb{E}_{t_n}^{X^n}[R_{zj}^{n+1}].
\end{equation}
Inserting $e_Y^{n+1}$ in \eqref{s3:e31b}
and $- \mathbb{E}_{t_n}^{X^n}[e_Z^{n+1}]$  in \eqref{s3:eyn131} into  \eqref{s3:e32z},
we deduce
\begin{equation}\label{231}
\begin{aligned}
e_Z^n=\;&
\mathbb{E}_{t_n}^{X^n}[e_Z^{n+2}]+\frac{2}{\Delta}\mathbb{E}_{t_n}^{X^n}[e_Y^{n+2}
\big(\Delta W_{n}^\top-\Delta W_{n+1}^\top\big)]\\&
+\mathbb{E}_{t_n}^{X^n}[e_f^{n+2}\Delta
W_{n}^\top]-\mathbb{E}_{t_n}^{X^n}[e_f^{n+2}\Delta W_{n+1}^\top]
+2\mathbb{E}_{t_n}^{X^n}[e_f^{n+1}\Delta W_{n}^\top]\\&+
 \sum\limits_{j=1}^2\frac{2}{\Delta}\Big(\mathbb{E}_{t_n}^{X^n}[R^{n+1}_{yj}\Delta W_{n}^\top]+R_{zj}^n-\mathbb{E}_{t_n}^{X^n}[R_{zj}^{n+1}]\Big).
\end{aligned}
\end{equation}
By the Malliavin integration-by-parts formula \eqref{Malp} and the chain rule \eqref{DY}, we have
\begin{equation}\label{5.12}\begin{aligned}&
\mathbb{E}_{t_n}^{X^n}[e_Y^{n+2}\big(\Delta W_{n}^\top-\Delta W_{n+1}^\top\big)]
=\mathbb{E}_{t_n}^{X^n}\Big[e_{\nabla Y}^{n+2}
\Big(\int_{t_n}^{t_{n+1}}D_tX^{n+2}dt-\int_{t_{n+1}}^{t_{n+2}}D_tX^{n+2}dt\Big)\Big],
\end{aligned}\end{equation}
where
$D_t X^{n+2}=\big[D_{1,t} X^{n+2},\ldots,D_{d,t} X^{n+2}\big]_{d\times d}$ is a $d\times d$ square matrix
 with $$D_{j, t}X^{n+2}=(D_{j, t} X_1^{n+2},\ldots,D_{j, t} X_d^{n+2})^\top \quad \textrm{for} \quad 1\leq j\leq d.$$
The weak order-2 It\^o-Taylor approximation solution $X^{n+2}$ can be represented as 
\begin{equation*}\label{Xn2}
\begin{aligned}
X^{n+2}
=\;&X^{n+1}+b(t_{n+1},X^{n+1})\Delta +\sigma(t_{n+1},X^{n+1})\Delta W_{n+1}\\
&+\sum_{\alpha\in \mathcal{A}_2}g_\alpha(t_{n+1},X^{n+1})I_{\alpha, n+1}.
\end{aligned}
\end{equation*}
Taking the Malliavin derivative $D_t$ to both sides of the above equation yields
\begin{equation}\label{2.35}
\int_{t_{n+1}}^{t_{n+2}}D_t X^{n+2}dt
=\sigma(t_{n+1},X^{n+1})\Delta +\sum_{\alpha\in \mathcal{A}_2}g_\alpha(t_{n+1},X^{n+1})\int_{t_{n+1}}^{t_{n+2}}D_tI_{\alpha,n+1}dt,
\end{equation}
and
\begin{equation}\label{2.36}
\begin{split}
\int_{t_n}^{t_{n+1}}D_tX^{n+2}dt
=\;&
\int_{t_n}^{t_{n+1}}D_tX^{n+1}dt+\sum_{\alpha\in \Gamma_2\backslash \{v\}}\int_{t_n}^{t_{n+1}}D_t\big\{g_\alpha(t_{n+1},X^{n+1})I_{\alpha, n+1}\big\}dt\\
=\;&
\sigma(t_n,X^n)\Delta +\sum_{\alpha\in \mathcal{A}_2}g_\alpha(t_{n},X^{n})\int_{t_{n}}^{t_{n+1}}D_tI_{\alpha, n}dt\\
\;&
+\sum_{\alpha\in \Gamma_2\backslash \{v\}}\int_{t_n}^{t_{n+1}}D_t\big\{g_\alpha(t_{n+1},X^{n+1})I_{\alpha, n+1}\big\}dt.
\end{split}
\end{equation}
By \eqref{5.12}, \eqref{2.35} and \eqref{2.36}, we deduce
\begin{equation}\label{5.15}
\begin{aligned}
&\mathbb{E}_{t_n}^{X^n}[e_Y^{n+2}\big(\Delta W_{n}^\top-\Delta W_{n+1}^\top\big)]\\
=\;&\mathbb{E}_{t_n}^{X^n}[e_{\nabla Y}^{n+2}\big(\sigma(t_n,X^n)-\sigma(t_{n+1},X^{n+1})\big)]\Delta
+\sum_{\alpha\in \mathcal{A}_2}\mathbb{E}_{t_n}^{X^n}[e_{\nabla Y}^{n+2}g_\alpha(t_{n},X^{n})\int_{t_{n}}^{t_{n+1}}D_tI_{\alpha, n}dt]
\\&
-\sum\limits_{\alpha\in \mathcal{A}_2}\mathbb{E}_{t_n}^{X^n}\big[e_{\nabla Y}^{n+2}g_\alpha(t_{n+1},X^{n+1})\int_{t_{n+1}}^{t_{n+2}}D_tI_{\alpha, n+1}dt\big]
\\&
+\sum\limits_{\alpha\in \Gamma_2\backslash \{v\}}\int_{t_n}^{t_{n+1}}\mathbb{E}_{t_n}^{X^n}\Big[e_{\nabla Y}^{n+2}D_t\big\{g_\alpha(t_{n+1},X^{n+1})I_{\alpha, n+1}\big\}\Big]dt.
\end{aligned}
\end{equation}
Also,
\begin{equation}\label{5.16}
\begin{aligned}
&\mathbb{E}_{t_n}^{X^n}[e_f^{n+1}\Delta W_{n}^\top]=\int_{t_n}^{t_{n+1}}\mathbb{E}_{t_n}^{X^n}[D_te_f^{n+1}]dt=\int_{t_n}^{t_{n+1}}\mathbb{E}_{t_n}^{X^n}[e_{\nabla f}^{n+1}D_tX^{n+1}]dt,\\
&\mathbb{E}_{t_n}^{X^n}[e_f^{n+2}\Delta W_{n}^\top]=\int_{t_n}^{t_{n+1}}\mathbb{E}_{t_n}^{X^n}[e_{\nabla f}^{n+2}D_tX^{n+2}]dt,\\
&\mathbb{E}_{t_n}^{X^n}[e_f^{n+2}\Delta W_{n+1}^\top]=\int_{t_{n+1}}^{t_{n+2}}\mathbb{E}_{t_n}^{X^n}[e_{\nabla f}^{n+2}D_tX^{n+2}]dt.
\end{aligned}
\end{equation}
Under the assumption $b,\sigma\in C_b^{1,2}$ and Proposition \ref{hyp2},  using the equalities in \eqref{5.15}, \eqref{5.16}  and the H\"older inequality, we have
\begin{equation}\label{5.17}
\begin{aligned}
&|\mathbb{E}_{t_n}^{X^n}[e_Y^{n+2}\big(\Delta W_{n}^\top-\Delta W_{n+1}^\top\big)]|^2\leq C\Delta^3\mathbb{E}_{t_n}^{X^n}\big[|e_{\nabla Y}^{n+2}|^2\big],\\
&|\mathbb{E}_{t_n}^{X^n}[e_f^{n+1}\Delta W_{n}^\top]|^2\leq C \Delta^2\mathbb{E}_{t_n}^{X^n}\big[|e_{Y}^{n+1}|^2+|e_{Z}^{n+1}|^2+|e_{\nabla Y}^{n+1}|^2+|e_{\nabla Z}^{n+1}|^2\big],\\
&|\mathbb{E}_{t_n}^{X^n}[e_f^{n+2}\Delta W_{n}^\top]|^2\leq C \Delta^2\mathbb{E}_{t_n}^{X^n}\big[|e_{Y}^{n+2}|^2+|e_{Z}^{n+2}|^2+|e_{\nabla Y}^{n+2}|^2+|e_{\nabla Z}^{n+2}|^2\big],\\
&|\mathbb{E}_{t_n}^{X^n}[e_f^{n+2}\Delta W_{n+1}^\top]|^2\leq C \Delta^2\mathbb{E}_{t_n}^{X^n}\big[|e_{Y}^{n+2}|^2+|e_{Z}^{n+2}|^2+|e_{\nabla Y}^{n+2}|^2+|e_{\nabla Z}^{n+2}|^2\big].
\end{aligned}\end{equation}
  Similarly, by taking square
on both sides of the equation \eqref{231},  and using the inequalities  in \eqref{5.17} and
the Young's inequality again,
we obtain
\begin{equation}\label{s3:e64}
\begin{aligned}
|e_Z^n|^2 \leq \;&
(1+\gamma_2\Delta)|\mathbb{E}_{t_n}^{X^n}[e_Z^{n+2}]|^2+\frac{C}{\gamma_2}
(1+\gamma_2\Delta )\Big\{\Delta\mathbb{E}_{t_n}^{X^n}
\big[|e_Y^{n+1}|^2+|e_Z^{n+1}|^2\\&+|e_{\nabla Y}^{n+1}|^2
+|e_{\nabla Z}^{n+1}|^2+|e_Y^{n+2}|^2+|e_Z^{n+2}|^2+|e_{\nabla Y}^{n+2}|^2
+|e_{\nabla Z}^{n+2}|^2\big]\\
& +\sum\limits_{j=1}^2\frac{|\mathbb{E}_{t_n}^{X^n}
[R_{yj}^{n+1}\Delta W_{n}^\top]|^2+|R_{zj}^n
-\mathbb{E}_{t_n}^{X^n}[R_{zj}^{n+1}]|^2}{\Delta ^3}\Big\}.
\end{aligned}
\end{equation}

(3) {\bf The estimate of $e_{\nabla Y}^{n}$.}

Taking variation on both sides of the equation \eqref{eyn1} gives
\begin{equation}\label{s3:e31}
\begin{aligned}
e_{\nabla Y}^n=& \; \mathbb{E}_{t_n}^{X^n}[e_{\nabla Y}^{n+2}\nabla_{x^n} X^{n+2}]
+\frac{1}{2}\Delta  e_{\nabla f}^n+\Delta \mathbb{E}_{t_n}^{X^n}[e_{\nabla f}^{n+1}\nabla_{x^n} X^{n+1}]\\&
+\frac{1}{2}\Delta \mathbb{E}_{t_n}^{X^n}[e_{\nabla f}^{n+2}\nabla_{x^n} X^{n+2}]+\sum\limits_{j=1}^2\Big(\nabla_{x^n}R_{yj}^n+\mathbb{E}_{t_n}^{X^n}[\nabla_{x^n}R_{yj}^{n+1}]\Big).
\end{aligned}
\end{equation}
Using the Taylor expansion to $\sigma(t_{n+1},X^{n+1})$, we have
\begin{equation*}
\sigma(t_{n+1},X^{n+1})=\sigma(t_{n+1},X^{n})+\int_0^1\sigma_x(t_{n+1},X^n+\lambda(X^{n+1}-X^n)) (X^{n+1}-X^n)\, d\lambda,
\end{equation*}
which combining the equation \eqref{Xn2} implies
\begin{equation}\label{s3:e31a1aa}\begin{array}{rl}
X^{n+2}
= X^{n}+\sigma(t_n, X^n)\Delta W_n+\sigma(t_{n+1},X^{n})\Delta W_{n+1}+\Lambda_n,
\end{array}\end{equation}
where
\begin{equation}\label{1}
\begin{aligned}
\Lambda_n=
&\;b(t_{n},X^{n})\Delta
+\sum\limits_{\alpha\in\mathcal{A}_2}g_\alpha(t_{n},X^{n}) I_{\alpha, n}\\&
+\,b(t_{n+1},X^{n+1})\Delta
+\sum\limits_{\alpha\in \mathcal{A}_2}g_\alpha(t_{n+1},X^{n+1})I_{\alpha, n+1}\\&
+\int_0^1\sigma_x(t_{n+1},X^n+\lambda(X^{n+1}-X^n)) (X^{n+1}-X^n)\Delta W_{n+1}\, d\lambda.
\end{aligned}
\end{equation}
Taking variation on both sides of the equation \eqref{s3:e31a1aa} gives
\begin{equation}\label{s3:e31a}\begin{aligned}
\nabla_{x^n} X^{n+2}
=& \quad I_{d\times d}+\sigma_x(t_{n},X^{n})\Delta W_{n}+\sigma_x(t_{n+1},X^{n})\Delta W_{n+1}+\nabla_{x^n}\Lambda_n.
\end{aligned}\end{equation}
Then,
\begin{equation}\label{s3525}
\begin{aligned}
\mathbb{E}_{t_n}^{X^n}[e_{\nabla Y}^{n+2}\nabla_{x^n} X^{n+2}]
=\;&\mathbb{E}_{t_n}^{X^n}[e_{\nabla Y}^{n+2}]+\sigma_x(t_{n},X^{n})\mathbb{E}_{t_n}^{X^n}[e_{\nabla Y}^{n+2}\Delta W_{n}^\top]\\&
+\sigma_x(t_{n+1},X^{n})\mathbb{E}_{t_n}^{X^n}[e_{\nabla Y}^{n+2}\Delta W_{n+1}^\top]+\mathbb{E}_{t_n}^{X^n}[e_{\nabla Y}^{n+2}\nabla_{x^n}\Lambda_n].
\end{aligned}
\end{equation}
By the H\"older inequality, under the assumption $b,\sigma\in C_b^{1,3}$,
we get
\begin{equation}\label{526}
\begin{aligned}
&\mathbb{E}_{t_n}^{X^n}[|\nabla_{X^n} \Lambda_n|^2]\leq C(1+|X^{n}|^{q})\Delta^2,\\
&\big|\sigma_x(t_{n},X^{n})\mathbb{E}_{t_n}^{X^n}
[e_{\nabla Y}^{n+2}\Delta W_{n}^\top]\big|^2
\leq C
\Delta Var^n(e_{\nabla Y}^{n+2}),  \\
&\big|\sigma_x(t_{n+1},X^{n})\mathbb{E}_{t_n}^{X^n}[e_{\nabla Y}^{n+2}
 \Delta W_{n+1}^\top]\big|^2\leq C\Delta Var^n(e_{\nabla Y}^{n+2}).
\end{aligned}
\end{equation}
Taking square on both sides of the equation \eqref{s3:e31} and using Young's inequality $\ds (a+b)^2\leq (1+\gamma_3\Delta )a^2+(1+\frac{1}{\gamma_3\Delta })b^2$ (for any $\gamma_3> 0$), we obtain
\begin{equation}\label{s3:349}\begin{aligned}
|e_{\nabla Y}^n|^2
\leq \; & (1+\gamma_3\Delta )|\mathbb{E}_{t_n}^{X^n}[e_{\nabla Y}^{n+2}]|^2
+\frac{C}{\gamma_3}(1+\gamma_3\Delta )\Big\{\Delta \big(|e_Y^{n}|^2+|e_Z^{n}|^2+|e_{\nabla Y}^{n}|^2+|e_{\nabla Z}^{n}|^2\big) \\&
+Var^n(e_{\nabla Y}^{n+2})
+\Delta \mathbb{E}_{t_n}^{X^n}\big[|e_Y^{n+1}|^2+|e_Z^{n+1}|^2+|e_{\nabla Y}^{n+1}|^2+|e_{\nabla Z}^{n+1}|^2+|e_Y^{n+2}|^2
\\&+|e_Z^{n+2}|^2+|e_{\nabla Y}^{n+2}|^2+|e_{\nabla Z}^{n+2}|^2\big]
+\sum\limits_{j=1}^2\frac{|\nabla_{x^n} R_{yj}^n|^2+|\mathbb{E}_{t_n}^{X^n}[\nabla_{x^{n}} R_{yj}^{n+1}]|^2}{\Delta }\Big\}.
\end{aligned}\end{equation}

(4) {\bf The estimate of $e_{\nabla Z}^{n}$.}

Taking variation on both sides of the equation \eqref{231} leads to
\begin{equation}\label{s3.57}\begin{aligned}
\Delta  e_{\nabla Z}^n
=&\;\Delta \mathbb{E}_{t_n}^{X^n}[e_{\nabla Z}^{n+2}\nabla_{x^n} X^{n+2}]
-2\mathbb{E}_{t_n}^{X^n}[\Delta W_{n+1}e_{\nabla Y}^{n+2}\nabla_{x^n} X^{n+2}]\\&
+2\mathbb{E}_{t_n}^{X^n}[\Delta W_{n}e_{\nabla Y}^{n+2}\nabla_{x^n} X^{n+2}]
+\Delta \mathbb{E}_{t_n}^{X^n}[\Delta W_{n} e_{\nabla f}^{n+2}\nabla_{x^n} X^{n+2}]\\&
-\Delta \mathbb{E}_{t_n}^{X^n}[\Delta W_{n+1} e_{\nabla f}^{n+2}\nabla_{x^n} X^{n+2}]
+2\Delta \mathbb{E}_{t_n}^{X^n}[\Delta W_{n} e_{\nabla f}^{n+1}\nabla_{x^n} X^{n+1}]\\
&+
2\sum\limits_{j=1}^2\big(\mathbb{E}_{t_n}^{X^n}[\Delta W_{n}\nabla_{x^n} R_{yj}^{n+1}]+\nabla_{x^n}R_{zj}^n-\mathbb{E}_{t_n}^{X^n}[\nabla_{x^n}R_{zj}^{n+1}]\big).
\end{aligned}\end{equation}
By the facts $\nabla_{x^n} X^{n+1}=I_{d\times d}+\sum\limits_{\alpha\in \Gamma_2 \backslash \{v\}}\partial_xg_\alpha(t_n, X^n)I_{\alpha, n}$ and
\begin{equation*}
\nabla_{x^n} X^{n+2}=\nabla_{x^{n+1}} X^{n+2}\nabla_{x^n} X^{n+1}=\Big(I_{d\times d}+\sum\limits_{\alpha\in \Gamma_2 \backslash \{v\}}\partial_xg_\alpha(t_{n+1}, X^{n+1})I_{\alpha,n+1}\Big)\nabla_{x^n} X^{n+1},
\end{equation*}
and using the H\"older inequality, we have the estimate
\begin{equation}\label{s3.58}\begin{aligned}\nonumber
&\,|\mathbb{E}_{t_n}^{X^n}[\Delta W_{n}e_{\nabla Y}^{n+2}\nabla_{x^n} X^{n+2}]|^2\\
\leq &\, 2|\mathbb{E}_{t_n}^{X^n}[\Delta W_{n}e_{\nabla Y}^{n+2}\nabla_{x^n} X^{n+1}]|^2\\&
+2\Big|\mathbb{E}_{t_n}^{X^n}\Big[\Delta W_{n}e_{\nabla Y}^{n+2}\sum\limits_{\alpha\in \Gamma_2 \backslash \{v\}}\partial_xg_\alpha(t_{n+1}, X^{n+1})I_{\alpha,n+1}
\nabla_{x^n} X^{n+1}\Big]\Big|^2\\
\leq &\, 2d\Delta Var^{n}(e_{\nabla Y}^{n+2}) +C\Delta^2\mathbb{E}_{t_n}^{X^n}[|e_{\nabla Y}^{n+2}|^2].
\end{aligned}\end{equation}
Similarly we have the estimates
\begin{equation}\label{s358}\begin{aligned}\nonumber
&|\mathbb{E}_{t_n}^{X^n}[\Delta W_{n+1}e_{\nabla Y}^{n+2}\nabla_{x^n} X^{n+2}]|^2
\leq 2d\Delta \mathbb{E}_{t_n}^{X^n}[Var^{n+1}(e_{\nabla Y}^{n+2})]
+C\Delta^2\mathbb{E}_{t_n}^{X^n}[|e_{\nabla Y}^{n+2}|^2],\\
&|\mathbb{E}_{t_n}^{X^n}[\Delta W_{n}e_{\nabla f}^{n+1}\nabla_{x^n} X^{n+2}]|^2
\leq C\Delta \mathbb{E}_{t_n}^{X^n}[|e_{Y}^{n+1}|^2+|e_Z^{n+1}|^2+|e_{\nabla Y}^{n+1}|^2+|e_{\nabla Z}^{n+1}|^2],\\
&|\mathbb{E}_{t_n}^{X^n}[\Delta W_{n}e_{\nabla f}^{n+2}\nabla_{x^n} X^{n+2}]|^2
\leq C\Delta \mathbb{E}_{t_n}^{X^n}[|e_{Y}^{n+2}|^2+|e_Z^{n+2}|^2+|e_{\nabla Y}^{n+2}|^2+|e_{\nabla Z}^{n+2}|^2],\\
&|\mathbb{E}_{t_n}^{X^n}[\Delta W_{n+1}e_{\nabla f}^{n+2}\nabla_{x^n} X^{n+2}]|^2
\leq C\Delta \mathbb{E}_{t_n}^{X^n}[|e_{Y}^{n+2}|^2+|e_Z^{n+2}|^2+|e_{\nabla Y}^{n+2}|^2
+|e_{\nabla Z}^{n+2}|^2].
\end{aligned}\end{equation}
Now by \eqref{s3.57}, the above five estimates, and using  $\big(\sum\limits_{i=1}^{m}a_i\big)^2\leq m\sum\limits_{i=1}^{m}a_i^2$, we deduce
\begin{equation}\label{s329}
\begin{aligned}
\Delta  |e_{\nabla Z}^n|^2
\leq \;& C\Delta |\mathbb{E}_{t_n}^{X^n}[e_{\nabla Z}^{n+2}\nabla_{x^n} X^{n+2}]|^2+
CVar^{n}(e_{\nabla Y}^{n+2})+C\mathbb{E}_{t_n}^{X^n}[Var^{n+1}(e_{\nabla Y}^{n+2})]\\&
+C\Delta \mathbb{E}_{t_n}^{X^n}[|e_{Y}^{n+2}|^2]
+C\Delta^2\mathbb{E}_{t_n}^{X^n}[|e_{Y}^{n+1}|^2+|e_Z^{n+1}|^2+|e_{\nabla Y}^{n+1}|^2+|e_{\nabla Z}^{n+1}|^2]\\&+C\Delta^2\mathbb{E}_{t_n}^{X^n}[
|e_{Y}^{n+2}|^2+|e_Z^{n+2}|^2+|e_{\nabla Y}^{n+2}|^2+|e_{\nabla Z}^{n+2}|^2]\\&
+\frac{C}{\Delta }\sum\limits_{j=1}^2\big(|\mathbb{E}_{t_n}^{X^n}[\Delta W_{n}\nabla_{x^n} R_{yj}^{n+1}]|^2+|\nabla_{x^n}R_{zj}^n|^2+|\mathbb{E}_{t_n}^{X^n}[\nabla_{x^n}R_{zj}^{n+1}]|^2\big).
\end{aligned}
\end{equation}
We remain to estimate the first term on the right side of \eqref{s329}.
Taking the variation $\nabla_{x^n}$ to $X^{n+2}$, which gives
\begin{equation*}
\begin{aligned}
\nabla_{x^n} X^{n+2}
=\,&
\nabla_{x^n}\Big[X^{n}+\sum\limits_{\alpha\in \Gamma_2\backslash \{v\}}g_\alpha(t_{n},X^{n})I_{\alpha, n}+\sum\limits_{\alpha \in\Gamma_2\backslash \{v\}}g_\alpha(t_{n+1},X^{n+1})I_{\alpha, n+1}\Big]\\
=\,&
I_{d\times d}+\sum\limits_{\alpha\in \Gamma_2\backslash \{v\}}\partial_x g_\alpha(t_{n},X^{n})I_{\alpha, n}+\sum\limits_{\alpha\in \Gamma_2 \backslash \{v\}}\partial_x g_\alpha(t_{n+1},X^{n+1})\nabla_{x^n} X^{n+1}I_{\alpha, n+1},
\end{aligned}
\end{equation*}
and using the inequality $\big(\sum\limits_{i=1}^{m}a_i\big)^2\leq m\sum\limits_{i=1}^ma_i^2$,
we get
\begin{equation}\label{s331}
\begin{aligned}
|\mathbb{E}_{t_n}^{X^n}[e_{\nabla Z}^{n+2}\nabla_{x^n} X^{n+2}]|^2
\leq \;& 3|\mathbb{E}_{t_n}^{X^n}[e_{\nabla Z}^{n+2}]|^2+C\Delta \mathbb{E}_{t_n}^{X^n}[|e_{\nabla Z}^{n+2}|^2].
\end{aligned}
\end{equation}
Using the integration-by-parts formula of Malliavin calculus \eqref{Malp}  we obtain
\begin{equation}\label{ez2W}
\begin{aligned}\hspace*{1.5cm}
\mathbb{E}_{t_{n+1}}^{X^{n+1}}\big[e_Z^{i,n+2}\Delta W_{n+1}^\top\big]
=
&\,\mathbb{E}_{t_{n+1}}^{X^{n+1}}\big[e_{\nabla Z}^{i,n+2}\big]\sigma(t_{n+1},X^{n+1})\Delta
\\&
+\sum\limits_{\alpha\in \mathcal{A}_2 }\mathbb{E}_{t_{n+1}}^{X^{n+1}}\big[e_{\nabla Z}^{i,n+2}g_\alpha(t_{n+1},X^{n+1})\int_{t_{n+1}}^{t_{n+2}} D_tI_{\alpha, n+1}dt\big].
\end{aligned}
\end{equation}
By the definition of the norm $|\cdot|$, we have
\begin{equation*}
|e_{\nabla Z}^{n+2}|^2 =\textrm{trace}\big([e_{\nabla Z}^{n+2}]^\top e_{\nabla Z}^{n+2}\big)=\sum\limits_{i,j=1}^d|e_{\nabla Z}^{i,j,n+2}|^2=\sum\limits_{i=1}^de_{\nabla Z}^{i,n+2}[e_{\nabla Z}^{i,n+2}]^\top,
\end{equation*}
where $e_{\nabla Z}^{n+2}:=[e_{\nabla Z}^{i,j,n+2}]_{d\times d}$, and $e_{\nabla Z}^{i,n+2}:=\big(e_{\nabla Z}^{i,1,n+2},\ldots,e_{\nabla Z}^{i,d,n+2}\big)_{1\times d }$ is the $i$-th row vector of $e_{\nabla Z}^{n+2}$.
The uniformly elliptic condition
\begin{equation*}
 \sigma(t_{n+1},X^{n+1})\sigma^\top(t_{n+1},X^{n+1})\geq \frac{1}{K'}I_{d\times d}
\end{equation*}
in Assumption \ref{hyp1} implies that $\sigma(t_{n+1},X^{n+1})\sigma^\top(t_{n+1},X^{n+1})-\frac{1}{K'}I_{d\times d}$ is a positive semi-definite matrix, that is
\begin{equation*}\label{2.46}
\mathbb{E}_{t_{n+1}}^{X^{n+1}}[e_{\nabla Z}^{i,n+2}]
\big(\sigma(t_{n+1},X^{n+1})\sigma^\top(t_{n+1},X^{n+1})-\frac{1}{K'}I_{d\times d} \big)
\mathbb{E}_{t_{n+1}}^{X^{n+1}}[e_{\nabla Z}^{i,n+2}]^\top\geq 0,
\end{equation*}
which yields
\begin{equation}\label{2.52}
\begin{aligned}
&\sum\limits_{i=1}^d |\mathbb{E}_{t_{n+1}}^{X^{n+1}}[e_{\nabla Z}^{i,n+2}]\sigma(t_{n+1},X^{n+1})|^2\\
=\;& \sum\limits_{i=1}^d \mathbb{E}_{t_{n+1}}^{X^{n+1}}
[e_{\nabla Z}^{i,n+2}]\sigma(t_{n+1},X^{n+1})\sigma^\top(t_{n+1},X^{n+1})
\mathbb{E}_{t_{n+1}}^{X^{n+1}}[e_{\nabla Z}^{i,n+2}]^\top\\
\geq \;& \frac{1}{K'}\sum\limits_{i=1}^d \mathbb{E}_{t_{n+1}}^{X^{n+1}}[e_{\nabla Z}^{i,n+2}]
\mathbb{E}_{t_{n+1}}^{X^{n+1}}[e_{\nabla Z}^{i,n+2}]^\top
=\frac{1}{K'}|\mathbb{E}_{t_{n+1}}^{X^{n+1}}\big[e_{\nabla Z}^{n+2}]|^2.
\end{aligned}
\end{equation}
Thanks to the Cauchy-Schwarz inequality, it holds that
\begin{equation}\label{2.53}
  \sum\limits_{i=1}^d|\mathbb{E}_{t_{n+1}}^{X^{n+1}}[e_Z^{i,n+2}\Delta W_{n+1}^\top]|^2\leq d\Delta \sum\limits_{i=1}^d Var^{n+1}(e_Z^{i,n+2})=d\Delta  Var^{n+1}(e_Z^{n+2}).
\end{equation}
Now, using \eqref{ez2W}, \eqref{2.52} and \eqref{2.53} we deduce
\begin{equation*}\label{2.54}\begin{aligned}
 \frac{1}{K'}\Delta^2|\mathbb{E}_{t_{n+1}}^{X^{n+1}}[e_{\nabla Z}^{n+2}]|^2
  \leq \;& 2\sum\limits_{i=1}^d\big|\mathbb{E}_{t_{n+1}}^{X^{n+1}}[e_Z^{i,n+2}\Delta W_{n+1}^\top]\big|^2\\&
 +2\sum\limits_{i=1}^d\sum\limits_{\alpha\in \mathcal{A}_2}\big|\mathbb{E}_{t_{n+1}}^{X^{n+1}}[e_{\nabla Z}^{i,n+2} g_\alpha(t_{n+1},X^{n+1})\int_{t_{n+1}}^{t_{n+2}} D_tI_{\alpha, n+1}dt]\big|^2\\
  \leq \;& 2d\Delta Var^{n+1}(e_{Z}^{n+2})+C\Delta^3 \mathbb{E}_{t_{n+1}}^{X^{n+1}}[|e_{\nabla Z}^{n+2}|^2] ,
\end{aligned}\end{equation*}
which implies
\begin{equation}\label{3.52a}
\Delta |\mathbb{E}_{t_{n+1}}^{X^{n+1}}[e_{\nabla Z}^{n+2}]|^2
\leq  2K'd Var^{n+1}(e_Z^{n+2})+C\Delta^2\mathbb{E}_{t_{n+1}}^{X^{n+1}}[|e_{\nabla Z}^{n+2}|^2].
\end{equation}
By the inequality $$|\mathbb{E}_{t_n}^{X^n}[e_{\nabla Z}^{n+2}]|^2
=|\mathbb{E}_{t_n}^{X^n}[\mathbb{E}_{t_{n+1}}^{X^{n+1}}[e_{\nabla Z}^{n+2}]]|^2
\leq \mathbb{E}_{t_n}^{X^n}[|\mathbb{E}_{t_{n+1}}^{X^{n+1}}[e_{\nabla Z}^{n+2}]|^2],$$
\eqref{s329} and \eqref{3.52a}, we obtain
\begin{equation}\label{s329a}
\begin{aligned}
\Delta  |e_{\nabla Z}^n|^2
\leq \,&
C\,\mathbb{E}_{t_n}^{X^n}[Var^{n+1}(e_Z^{n+2})]+
CVar^{n}(e_{\nabla Y}^{n+2})
+C\mathbb{E}_{t_n}^{X^n}[Var^{n+1}(e_{\nabla Y}^{n+2})]\\
&
+\,C\Delta \mathbb{E}_{t_n}^{X^n}[|e_{Y}^{n+1}|^2]
+\,C\Delta^2\mathbb{E}_{t_n}^{X^n}[
|e_{Y}^{n+1}|^2+|e_Z^{n+1}|^2+|e_{\nabla Y}^{n+1}|^2+|e_{\nabla Z}^{n+1}|^2]\\
&
+C\Delta^2\mathbb{E}_{t_n}^{X^n}[
|e_{Y}^{n+2}|^2+|e_Z^{n+2}|^2+|e_{\nabla Y}^{n+2}|^2+|e_{\nabla Z}^{n+2}|^2]
+\sum\limits_{j=1}^2\frac{C}{\Delta }\big(|\mathbb{E}_{t_n}^{X^n}[\Delta W_{n}\nabla_{x^n} R^{n+1}_{yj} ]|^2\\
&
+|\nabla_{x^n}R_{zj}^n|^2+|\mathbb{E}_{t_n}^{X^n}[\nabla_{x^n}R_{zj}^{n+1}]|^2\big).\\
\end{aligned}
\end{equation}

(5) {\bf The estimate \eqref{eq:thm1} in the theorem.}

Combining the inequalities \eqref{s3:e71}, \eqref{s3:e64},
\eqref{s3:349} and \eqref{s329a}, and applying Lemma \ref{th1}, we deduce
\begin{equation}\label{ch123}
\begin{array}{rll}
&\!\!\!\!
|e_Y^n|^2+|e_Z^n|^2+|e_{\nabla Y}^n|^2+\frac{1}{4C}\Delta |e_{\nabla Z}^n|^2\\
\leq & \!\!\!\!\!
(1+C\Delta)\mathbb{E}_{t_n}^{X^n}
\big[|\mathbb{E}_{t_{n+1}}^{X^{n+1}}[e_Y^{n+2}]|^2
+|\mathbb{E}_{t_{n+1}}^{X^{n+1}}[e_Z^{n+2}]|^2
+|\mathbb{E}_{t_{n}}^{X^{n}}[e_{\nabla Y}^{n+2}]|^2
\big]\\
&\!\!\!\!
+\Big(\frac{C}{\gamma_1}+\frac{C}{\gamma_2}+\frac{C}{\gamma_3}\Big)
\mathbb{E}_{t_n}^{X^n}[Var^{n+1}(e_{Y}^{n+2})]
+\frac{1}{4}\mathbb{E}_{t_n}^{X^n}[Var^{n+1}(e_Z^{n+2})]\\
&\!\!\!\!
+\big(\frac{1}{4}+C(\frac{1}{\gamma_3}+\Delta)\big)Var^{n}(e_{\nabla Y}^{n+2})
+\Big(\frac{C}{\gamma_2}+\frac{C}{\gamma_3}+\frac{1}{4}\Big)
 \mathbb{E}_{t_n}^{X^n}[Var^{n+1}(e_{\nabla Y}^{n+2})]\\
&\!\!\!\!
+C\Delta \big(|e_Y^n|^2+|e_Z^n|^2+|e_{\nabla Y}^n|^2+|e_{\nabla Z}^n|^2\big)
+C\Delta\mathbb{E}_{t_n}^{X^n}\big[|e_Y^{n+2}|^2+|e_Z^{n+2}|^2\\
&\!\!\!\!
+|e_{\nabla Y}^{n+2}|^2+|e_{\nabla Z}^{n+2}|^2\big]
+C\sum\limits_{j=1}^2\frac{1}{\Delta^3}
\Big\{|\mathbb{E}_{t_n}^{X^n}[R_{yj}^{n+1}\Delta W_{n}^\top]|^2+|R_{zj}^n-\mathbb{E}_{t_n}^{X^n}[R_{zj}^{n+1}]|^2\Big\}\\
&\!\!\!\!
+C\sum\limits_{j=1}^2\frac{1}{\Delta}
\Big\{|\mathbb{E}_{t_n}^{X^n}[R_{yj}^{n+1}]|^2
+|\mathbb{E}_{t_n}^{X^n}[\nabla_{x^n} R_{yj}^{n+1}]|^2
+|\mathbb{E}_{t_n}^{X^n}[\nabla_{x^n} R_{zj}^{n+1}]|^2\\& \qquad \qquad
+|\mathbb{E}_{t_n}^{X^n}[\Delta W_{n}\nabla_{x^n}R_{yj}^{n+1}]|^2+|R_{yj}^n|^2+|\nabla_{x^n} R_{yj}^{n}|^2+|\nabla_{x^n} R_{zj}^{n}|^2\\& \qquad \qquad
+\mathbb{E}_{t_n}^{X^n}[|R_{zj}^{n+1}|^2]
+\mathbb{E}_{t_n}^{X^n}[|\nabla_{x^{n+1}}R_{zj}^{n+1}|^2]
\Big\}.
\end{array}
\end{equation}
Notice that
\begin{equation*}\label{s3:e70}
\begin{aligned}
&|\mathbb{E}_{t_{n+1}}^{X^{n+1}}[e_Y^{n+2}]|^2+\frac{3}{8}Var^{n+1}(e_Y^{n+2})
=\frac{5}{8}|\mathbb{E}_{t_{n+1}}^{X^{n+1}}[e_Y^{n+2}]|^2+\frac{3}{8}\mathbb{E}_{t_{n+1}}^{X^{n+1}}[|e_Y^{n+2}|^2]\leq \mathbb{E}_{t_{n+1}}^{X^{n+1}}[|e_Y^{n+2}|^2],\\
&|\mathbb{E}_{t_{n+1}}^{X^{n+1}}[e_{Z}^{n+2}]|^2
 +\frac{1}{4}Var^{n+1}(e_{Z}^{n+2})
=\frac{3}{4}|\mathbb{E}_{t_{n+1}}^{X^{n+1}}[e_{Z}^{n+2}]|^2
 +\frac{1}{4}\mathbb{E}_{t_{n+1}}^{X^{n+1}}[|e_{Z}^{n+2}|^2]
\leq\mathbb{E}_{t_{n+1}}^{X^{n+1}}[|e_{Z}^{n+2}|^2],\\
&
|\mathbb{E}_{t_{n}}^{X^{n}}[e_{\nabla Y}^{n+2}]|^2
+(\frac{1}{4}+\varsigma)Var^{n}(e_{\nabla Y}^{n+2})
+\frac{1}{2}\mathbb{E}_{t_n}^{X^n}[Var^{n+1}(e_{\nabla Y}^{n+2})]
\\&\!\!
\leq (\frac{1}{4}+\varsigma)|\mathbb{E}_{t_{n}}^{X^{n}}[e_{\nabla Y}^{n+2}]|^2
+(\frac{3}{4}-\varsigma)\mathbb{E}_{t_{n}}^{X^{n}}[|\mathbb{E}_{t_{n+1}}^{X^{n+1}}
[e_{\nabla Y}^{n+2}]|^2]\\&
\qquad+(\frac{1}{4}+\varsigma)Var^{n}(e_{\nabla Y}^{n+2})
+\frac{1}{2}\mathbb{E}_{t_n}^{X^n}[Var^{n+1}(e_{\nabla Y}^{n+2})]\\
&\!\!
\leq \big((\frac{1}{4}+\varsigma)+\frac{1}{2}\big)\mathbb{E}_{t_{n}}^{X^{n}}[|e_{\nabla Y}^{n+2}|^2]
+(\frac{1}{4}-\varsigma)\mathbb{E}_{t_{n}}^{X^{n}}[|\mathbb{E}_{t_{n+1}}^{X^{n+1}}
[e_{\nabla Y}^{n+2}]|^2]
\leq \mathbb{E}_{t_{n}}^{X^{n}}[|e_{\nabla Y}^{n+2}|^2],
\quad 0< \varsigma< \frac{1}{4}.
\end{aligned}
\end{equation*}
Now let $\gamma_1=\gamma_2=\gamma_3=8C$ and  $\Delta_{0}$ be sufficient small
such that $0<C(\frac{1}{\gamma_{3}}+\Delta)<\frac{1}{4}$ for $\Delta<\Delta_{0}$.
Then, by inequality \eqref{ch123}, we deduce
\begin{equation}\label{com-eq}
\begin{array}{rll}\hspace*{1cm}
&\!\!\!\!
|e_Y^n|^2+|e_Z^n|^2+|e_{\nabla Y}^n|^2+\frac{1}{4C}\Delta |e_{\nabla Z}^n|^2\\
\leq &\!\!\!\!\!
\frac{1+C\Delta }{1-C\Delta }\mathbb{E}_{t_n}^{X^n}\big[|e_Y^{n+2}|^2+|e_Z^{n+2}|^2+|e_{\nabla Y}^{n+2}|^2
+\frac{1}{4C}\Delta |e_{\nabla Z}^{n+2}|^2\big]\\
&\!\!\!\!
+C\sum\limits_{j=1}^2\frac{1}{\Delta^3}
\Big\{|\mathbb{E}_{t_n}^{X^n}[R_{yj}^{n+1}\Delta W_{n}^\top]|^2+|R_{zj}^n-\mathbb{E}_{t_n}^{X^n}[R_{zj}^{n+1}]|^2\Big\}\\
&\!\!\!\!
+C\sum\limits_{j=1}^2\frac{1}{\Delta}
\Big\{|\mathbb{E}_{t_n}^{X^n}[R_{yj}^{n+1}]|^2
+|\mathbb{E}_{t_n}^{X^n}[\nabla_{x^n} R_{yj}^{n+1}]|^2
+|\mathbb{E}_{t_n}^{X^n}[\nabla_{x^n} R_{zj}^{n+1}]|^2\\& \qquad \qquad
+|\mathbb{E}_{t_n}^{X^n}[\Delta W_{n}\nabla_{x^n}R_{yj}^{n+1}]|^2+|R_{yj}^n|^2+|\nabla_{x^n} R_{yj}^{n}|^2+|\nabla_{x^n} R_{zj}^{n}|^2\\& \qquad \qquad
+\mathbb{E}_{t_n}^{X^n}[|R_{zj}^{n+1}|^2]
+\mathbb{E}_{t_n}^{X^n}[|\nabla_{x^{n+1}}R_{zj}^{n+1}|^2]
\Big\}\\
\leq & \!\!\!\!\!
e^{CT}\mathbb{E}_{t_n}^{X^n}[|e_Y^{N-1}|^2+|e_Z^{N-1}|^2+|e_{\nabla Y}^{N-1}|^2+\frac{1}{4C}\Delta |e_{\nabla Z}^{N-1}|^2]\\
&\!\!\!\!
+C\sum\limits_{i=n}^{N-2}\sum\limits_{j=1}^2\Big\{\frac{1}{\Delta^3}
\Big(|\mathbb{E}_{t_i}^{X^i}[R_{yj}^{i+1}\Delta W_{i}^\top]|^2+|R_{zj}^i-\mathbb{E}_{t_i}^{X^i}[R_{zj}^{i+1}]|^2\Big)\\
&\!\!\!\!
\hspace{2cm}+\frac{1}{\Delta}
\Big(|\mathbb{E}_{t_i}^{X^i}[R_{yj}^{i+1}]|^2
+|\mathbb{E}_{t_i}^{X^i}[\nabla_{x^i} R_{yj}^{i+1}]|^2
+|\mathbb{E}_{t_i}^{X^i}[\nabla_{x^i} R_{zj}^{i+1}]|^2\\& \hspace{2.5cm}
+|\mathbb{E}_{t_i}^{X^i}[\Delta W_{i}\nabla_{x^i}R_{yj}^{i+1}]|^2+|R_{yj}^i|^2+|\nabla_{x^i} R_{yj}^{i}|^2+|\nabla_{x^i} R_{zj}^{i}|^2\\& \hspace{2.5cm}
+\mathbb{E}_{t_i}^{X^i}[|R_{zj}^{i+1}|^2]
+\mathbb{E}_{t_i}^{X^i}[|\nabla_{x^{i+1}}R_{zj}^{i+1}|^2]
\Big)\Big\}.
\end{array}
\end{equation}
And then by taking the mathematical expectation $\mathbb{E}[\cdot]$ on both sides of \eqref{com-eq},
we complete the proof.
\end{proof}

\begin{rem}
The remainder terms include three types:
(1) the truncation error terms, e.g., $R_{y1}^i$ and $\nabla_{x^i} R_{y1}^i$;
(2) the discretization errors caused by the discretization of SDE, e.g., $R_{y2}^i$ and $\nabla_{x^i} R_{y2}^i$;
(3) the error terms in (2) multiplied by $\Delta W_i^\top$
(e.g., $\mathbb{E}_{t_i}^{X^i}[R_{yj}^{i+1}\Delta W_{i}^\top]$).
Under certain regularity conditions on the data $b$, $\sigma$, $f$ and $\varphi$,
by the It\^o-Taylor and Taylor expansion, and the Malliavin calculus,
we can obtain the estimates of these remainder terms
(which are proved in detail in Section \ref{Sec:pf42}).
Subsequently,
it is easy to get error estimates for Scheme \ref{sch1} by Theorem \ref{thm1}.
\end{rem}

\subsection{Proof of Theorem \ref{thm2}}\label{Sec:pf42}

We consider the case that the generator $f$ of FBSDEs \eqref{DFBSDEs1} is a deterministic function.

\subsubsection{Useful lemmas}
In this subsection, we introduce some lemmas
which will be used in the proof of Theorem \ref{thm2}.
They may also be very useful in error analysis
for other numerical methods for solving FBSDEs.

\begin{lem}\label{lem51}
For $X^{n+1}=\sum\limits_{\alpha\in \Gamma_2}g_\alpha(t_n,X^n)I_{\alpha, n}$, if
$b,\sigma\in C_b^{2,5}$ and $H\in C_b^{5}$,
then under Hypothesis \ref{hyp0}, for $1\leq n\leq N-2$,
there exists a positive integer $q$ such that
\begin{equation}\label{lem35a}
\big|\nabla_{x^n}\mathbb{E}_{t_n}^{X^n}\big[H(X_{t_{n+1}}^{t_n,X^n})- H(X^{n+1})\big]\big|\leq  C(1+|X^n|^q)\Delta^3,
\end{equation}
where $C$ is a positive constant depending on $K$, and upper bounds of
the derivatives of $b$, $\sigma$ and $H$.
\end{lem}
\begin{proof}
For $0\leq n\leq N-2$, using the multiple Taylor expansion, we obtain
\begin{equation}\label{lem351a}
 \mathbb{E}_{t_{n}}^{X^{n}}[H(X_{t_{n+1}}^{t_n,X^n})-H(X^{n+1})]
=\sum\limits_{i=1}^d \mathbb{E}_{t_{n}}^{X^{n}}[h_{n}^i F_{x_{i}}^{X^n}],
\end{equation}
where $h_{n}^i=\sum\limits_{\alpha\in \mathcal{A}_3}I_\alpha[g_{\alpha}^i(\cdot,X_\cdot^{t_{n},X^{n}})]_{t_{n},t_{n+1}}$ and $F_{x_{i}}^{X^n}=\int_0^1 H_{x_i}'\big(X^{n+1}+\lambda(X_{t_{n+1}}^{t_n,X^n}-X^{n+1})\big)d\lambda.$
Then under the assumption $b,\sigma\in C_b^{2,5}$ and $H\in C_b^{5}$,
by the integration-by-parts formula \eqref{Malp} of Malliavin calculus
and inequality \eqref{est_Xs}, we deduce
\begin{equation*}
\begin{aligned}
&
\big|\nabla_{x^n} \mathbb{E}_{t_{n}}^{X^{n}}[H(X_{t_{n+1}}^{t_n,X^n})-H(X^{n+1})]\big|= \big|\sum\limits_{\alpha\in \mathcal{A}_3}\sum\limits_{i=1}^d \nabla_{x^n} \mathbb{E}_{t_{n}}^{X^{n}}[h_{n}^i F_{x_{i}}^{X^n}]\big|\\
=&
\big|\sum\limits_{\alpha\in \mathcal{A}_3}\sum\limits_{i=1}^d\int_{t_n}^{t_{n+1}}\int_{t_n}^{s_3}\int_{t_n}^{s_2} \mathbb{E}_{t_n}^{X^n}\big[\nabla_{x^n} \big\{D_{s_1 s_2 s_3}^{\alpha}(F_{x_{i}}^{X^{n}}) g_\alpha^i(s_1,X_{s_1}^{t_n,X^n})\big\}\big]ds_1ds_2 ds_3\big|\\
\leq \,&
C(1+|X^n|^q) \Delta^3.
\end{aligned}
\end{equation*}
The proof is competed.
\end{proof}

\begin{lem}\label{lem52}
If 
$b,\sigma\in C_b^{2,4}$ and $H\in C_{b}^{3,6}$,
then under Hypothesis \ref{hyp0}, for $1\leq n \leq N-2$,
there exists a generic positive integer $q$ such that

\begin{equation}\label{lem1eq1}
\big|\mathbb{E}_{t_{n}}^{X^n}\big[\Delta W_{n}R_{n+1}^{n+2}\big]\big|
\leq  C(1+|X^n|^q) \Delta^4,\\
\end{equation}
moreover, if $b,\sigma\in C_b^{2,5}$ and $H\in C_{b}^{3,7}$, then
\begin{equation}\label{lem1eq2}
\big|\mathbb{E}_{t_{n}}^{X^n}\big[\Delta W_{n} \nabla_{x^n} R_{n+1}^{n+2}\big]\big|
\leq  C(1+|X^n|^q)\Delta^4.
\end{equation}
where
$R_{n+1}^{n+2}=\mathbb{E}_{t_{n+1}}^{X^{n+1}}\big[\int_{t_{n+1}}^{t_{n+2}}\big\{H(t,X_t^{t_{n+1},X^{n+1}})
-\frac{H(t_{n+1},X^{n+1})+H(t_{n+2},X_{t_{n+2}}^{t_{n+1},X^{n+1}})}{2}\big\}\,dt\big],
$
and $C$ is a positive constant depending on $K$, and upper bounds of
the derivatives of $b$, $\sigma$ and $H$.
\end{lem}
\begin{proof}
Since $\Delta W_{n}$ is $\mathcal{F}_{t_{n+1}}$-measurable increment, we have the
identity
\begin{equation}\label{e4:e29b}\
\begin{array}{ll}
&\!\!\!\! \mathbb{E}_{t_{n}}^{X^n}\Big[\Delta
W_{n}\int_{t_{n+1}}^{t_{n+2}}\Big\{H(t,X_t^{t_{n+1},X^{n+1}})
-\frac{H(t_{n+1},X^{n+1})+H(t_{n+2},X_{t_{n+2}}^{t_{n+1},X^{n+1}})}{2}\Big\}\,dt\Big]\\
= &\!\!\!\! \mathbb{E}_{t_{n}}^{X^n}\Big[\Delta
W_{n}\mathbb{E}^{X^{n+1}}_{t_{n+1}}\Big[\int^{t_{n+2}}_{t_{n+1}}\Big\{H(t,X_t^{t_{n+1},X^{n+1}})
-\frac{H(t_{n+1},X^{n+1})+H(t_{n+2},X_{t_{n+2}}^{t_{n+1},X^{n+1}})}{2}\Big\}\,dt\Big]\Big].
\end{array}
\end{equation}
The It\^o formula then shows that
\begin{equation}\label{e4:e29}\begin{array}{rl}
H(t, X_t^{t_{n+1},X^{n+1}})=
&\!\!\!\!
H(t_{n+1}, X^{n+1})+\int^t_{t_{n+1}}L^0H(s, X_s^{t_{n+1},X^{n+1}}) ds\\
&\!\!\!\!
+\sum\limits_{j=1}^d\int^t_{t_{n+1}}L^jH(s, X_s^{t_{n+1},X^{n+1}})\,dW_s^j,\\
L^0H(s, X_s^{t_{n+1},X^{n+1}})
= &\!\!\!\! L^0H(t_{n+1},X^{n+1})+
\int^s_{t_{n+1}}L^0L^0H(\tau,  X_\tau^{t_{n+1},X^{n+1}})\,d\tau \\
&\!\!\!\!
+\sum\limits_{j=1}^d\int^s_{t_{n+1}}L^jL^0H(\tau,X_\tau^{t_{n+1},X^{n+1}})\,dW_\tau^j,\\
L^0L^0H(\tau, X_\tau^{t_{n+1},X^{n+1}}) = &\!\!\!\! L^0L^0H(t_{n+1},X^{n+1})
+\int_{t_{n+1}}^\tau L^0L^0L^0H(\nu, X_\nu^{t_{n+1},X^{n+1}})d\nu\\
&\!\!\!\!
+\sum\limits_{j=1}^d\int_{t_{n+1}}^\tau L^jL^0L^0H(\nu, X_\nu^{t_{n+1},X^{n+1}})dW_\nu^j.
\end{array}\end{equation}
By the equalities in \eqref{e4:e29}, we have
\begin{equation*}\label{e4:e32}
\begin{aligned}
&
\mathbb{E}^{X^{n+1}}_{t_{n+1}}\Big[\int^{t_{n+2}}_{t_{n+1}}H(t, X_t^{t_{n+1},X^{n+1}})\,dt\Big]\\
= \,&
H(t_{n+1}, X^{n+1})\Delta +\frac{1}{2}L^0H(t_{n+1},X^{n+1})\Delta^2+\frac{1}{6}L^0L^0H(t_{n+1},X^{n+1})\Delta^3\\
&
+\int_{t_{n+1}}^{t_{n+2}}\int_{t_{n+1}}^t\int_{t_{n+1}}^s\int_{t_{n+1}}^\tau\mathbb{E}_{t_{n+1}}^{X^{n+1}}[L^0L^0L^0H(\nu, X_\nu^{t_{n+1},X^{n+1}})]\, d\nu d\tau ds dt,\\
\end{aligned}
\end{equation*}
and
\begin{equation*}\label{e4:e33}
\begin{aligned}
&
-\frac{1}{2}\int_{t_{n+1}}^{t_{n+2}}\mathbb{E}_{t_{n+1}}^{X^{n+1}}[H(t_{n+2},X_{t_{n+2}}^{t_{n+1},X^{n+1}})]\,dt\\
=&
-\frac{1}{2}H(t_{n+1}, X^{n+1})\Delta -\frac{1}{2}L^0H(t_{n+1}, X^{n+1})\Delta^2-\frac{1}{4}L^0L^0H(t_{n+1},X^{n+1})\Delta^3\\
&
-\frac{1}{2}\int_{t_{n+1}}^{t_{n+2}}\int_{t_{n+1}}^{t_{n+2}}\int_{t_{n+1}}^s\int_{t_{n+1}}^\tau\mathbb{E}_{t_{n+1}}^{X^{n+1}}[L^0L^0L^0H(\nu,
X_\nu^{t_{n+1},X^{n+1}})]\, d\nu d\tau  ds dt.
\end{aligned}
\end{equation*}
Then by the above two identities, we deduce
\begin{equation}\label{s7:i19}\begin{array}{ll}
&~\mathbb{E}_{t_{n}}^{X^n}\big[\Delta
W_{n}R_{n+1}^{n+2}\big]\\
=&-\frac{1}{12}\mathbb{E}_{t_{n}}^{X^{n}}\big[\Delta W_{n}L^0L^0H(t_{n+1},X^{n+1})\big]\Delta^3\\&
+\int^{t_{n+2}}_{t_{n+1}}\int^t_{t_{n+1}}\int^s_{t_{n+1}}\int^\tau_{t_{n+1}}\mathbb{E}_{t_{n}}^{X^{n}}[\Delta W_{n}L^0L^0L^0H(\nu, X_\nu^{t_{n+1},X^{n+1}})]\, d\nu d\tau  ds dt\\&
-\frac{1}{2}\int_{t_{n+1}}^{t_{n+2}}\int_{t_{n+1}}^{t_{n+2}}\int_{t_{n+1}}^s\int_{t_{n+1}}^\tau\mathbb{E}_{t_{n}}^{X^{n}}[\Delta W_{n}L^0L^0L^0H(\nu,
X_\nu^{t_{n+1},X^{n+1}})]\, d\nu d\tau  ds dt.
\end{array}\end{equation}
From the Malliavin integration-by-parts formula \eqref{Malp} we deduce
\begin{equation}\label{311}
\mathbb{E}_{t_{n}}^{X^{n}}[\Delta W_{n}L^0L^0H(t_{n+1},X^{n+1})]=\int_{t_n}^{t_{n+1}}\mathbb{E}_{t_{n}}^{X^{n}}\big[D_sL^0L^0H(t_{n+1},X^{n+1})]ds.
\end{equation}
Now, under the assumptions of the lemma, combining (\ref{e4:e29b}), (\ref{s7:i19}) and \eqref{311}
we easily obtain the inequality  \eqref{lem1eq1}.
Similarly we can prove the inequality \eqref{lem1eq2}.  The proof is completed.
\end{proof}

\begin{lem}\label{lem53}
Suppose 
$b,\sigma\in C_b^{2,4}$ and $H\in C_b^{3,6}$,
then under Hypotheses \ref{hyp0},
for $0\leq n\leq N-2$, there exists a positive integer $q$ such that
\begin{equation*}\begin{array}{rr}&
\big|\mathbb{E}_{t_n}^{X^n}\big[ R_{n}^{n+1}- R_{n+1}^{n+2}\big]\big|\leq
C(1+|X^n|^q)\Delta^4,
\end{array}\end{equation*}
where
$
  R_{n}^{n+1}=\mathbb{E}_{t_{n}}^{X^{n}}\big[\int_{t_{n}}^{t_{n+1}}\big\{H(t,X_t^{t_{n},X^{n}})
-\frac{H(t_{n},X^{n})+H(t_{n+1},X_{t_{n+1}}^{t_{n},X^{n}})}{2}\big\}\,dt\big],
$
and $C$ is a positive constant depending on $K$, and upper bounds of
the derivatives of $b,\sigma$ and $H$.
\end{lem}
\begin{proof} Similar  to get (\ref{s7:i19}), we have the following two
equalities:
\begin{equation*}\label{e4:e10}
\begin{array}{ll}
&\!\!\!\!
\mathbb{E}_{t_n}^{X^n}[R_n^{n+1}]\\
=&\!\!\!\!
-\frac{1}{12}L^0L^0H(t_n,X^n)\Delta^3
+\int_{t_n}^{t_{n+1}}\int_{t_n}^t\int_{t_n}^s\int_{t_n}^r\mathbb{E}_{t_n}^{X^n}[L^0L^0L^0H(\tau, X_\tau^{t_n,X^n})]\,d\tau dr ds dt\\
&\!\!\!\!
-\frac{1}{2}\int_{t_n}^{t_{n+1}}\int_{t_n}^{t_{n+1}}\int_{t_n}^s\int_{t_n}^r\mathbb{E}_{t_n}^{X^n}[L^0L^0L^0H(\tau,
X_\tau^{t_n,X^n})]\,d\tau dr ds dt,
\end{array}
\end{equation*}
and
\begin{equation*}\label{e4:e11}
\begin{array}{ll}
&\!\!\!\! \mathbb{E}_{t_{n+1}}^{X^{n+1}}[R_{n+1}^{n+2}]\\
=&\!\!\!\!
-\frac{1}{12}L^0L^0H(t_{n+1},X^{n+1})\Delta^3
+\int_{t_{n+1}}^{t_{n+2}}\int_{t_{n+1}}^t\int_{t_{n+1}}^s\int_{t_{n+1}}^r
\mathbb{E}_{t_{n+1}}^{X^{n+1}}[L^0L^0L^0H(\tau, X_\tau^{t_{n+1},X^{n+1}})]\,d\tau dr ds dt\\
&\!\!\!\!
-\frac{1}{2}\int_{t_{n+1}}^{t_{n+2}}\int_{t_{n+1}}^{t_{n+2}}\int_{t_{n+1}}^s\int_{t_{n+1}}^r
\mathbb{E}_{t_{n+1}}^{X^{n+1}}[L^0L^0L^0H(\tau,
X_\tau^{t_{n+1},X^{n+1}})]\,d\tau dr ds dt.
\end{array}
\end{equation*}
Now, under the conditions of the lemma and  from the above two equations, we deduce
\begin{equation*}
\begin{aligned}
\quad\;\big|\mathbb{E}_{t_n}^{X^n}[ R_{n}^{n+1}- R_{n+1}^{n+2}]\big|
&
\leq\! \frac{1}{12}\Delta^3\big|\mathbb{E}_{t_n}^{X^n}\big[L^0L^0H(t_{n+1},X^{n+1})
-L^0L^0H(t_n,X^n)\big]\big|
+C(1+|X^n|^q)\Delta^4
\\&
\leq  C(1+|X^n|^q)\Delta^4.
\end{aligned}
\end{equation*}
We complete the proof.
\end{proof}

\begin{lem}\label{lem54}
For $X^{n+1}=\sum\limits_{\alpha\in\Gamma_2}g_\alpha(t_n,X^n)I_{\alpha, n}$,
if 
$b,\sigma\in C_b^{2,4}$ and $H\in C_b^{5}$.
then under Hypothesis \ref{hyp0},
for $1\leq n \leq N-2$, there exists a positive generic integer $q$ such that
\begin{eqnarray}\label{lemleq4}
\big|\mathbb{E}_{t_{n}}^{X^{n}}\big[\Delta W_{n} U_{n+1}^{n+2}\big]\big|
\leq  C(1+|X^n|^q) \Delta^4,
\end{eqnarray}
moreover, if $b,\sigma\in C_b^{2,5}$ and $H\in C_b^{6}$, then
\begin{eqnarray}\label{lemqel}
\big|\mathbb{E}_{t_{n}}^{X^{n}}\big[\Delta W_{n} \nabla_{x^n} U_{n+1}^{n+2}\big]\big|
\leq  C(1+|X^n|^q) \Delta^4,
\end{eqnarray}
where $U_{n+1}^{n+2}=\mathbb{E}_{t_{n+1}}^{X^{n+1}}\big[H(X_{t_{n+2}}^{t_{n+1},X^{n+1}})
-H(X^{n+2})]$,
and $C$ is a positive constant depending on $K$, and upper bounds of
the derivatives of $b,\sigma$ and $H$.
\end{lem}
\begin{proof}
The Taylor expansion shows that
\begin{equation*}\label{2.71}
H(X_{t_{n+2}}^{t_{n+1},X^{n+1}})-H(X^{n+2})
=\sum\limits_{i=1}^d h_{n+1}^i F_{x_i}^{X^{n+1}},
\end{equation*}
where $h_{n+1}^i=\sum\limits_{\alpha\in \mathcal{A}_3}I_\alpha[g_{\alpha}^i(\cdot,X_\cdot^{t_{n+1},X^{n+1}})]_{t_{n+1},t_{n+2}}$ and $F_{x_i}^{X^{n+1}}=\int_0^1 H_{x_i}'\big(X^{n+2}+\lambda(X_{t_{n+2}}^{t_{n+1},X^{n+1}}-X^{n+2})\big)d\lambda.$
Now, under the conditions of the lemma, using the integration-by-parts formula of Malliavin
calculus \eqref{Malp}, we have
\begin{equation*}\begin{array}{rll}
 U_{n+1}^{n+2}=&\!\!\!\!\mathbb{E}_{t_{n+1}}^{X^{n+1}}\big[H(X_{t_{n+2}}^{t_{n+1},X^{n+1}})-H(X^{n+2})\big]\\
=&\!\!\!\!\sum\limits_{\alpha\in \mathcal{A}_3} \sum\limits_{i=1}^d \mathbb{E}_{t_{n+1}}^{X^{n+1}}\big[ F_{x_i}^{X^{n+1}} I_\alpha[g_{\alpha}^i(\cdot,X_\cdot^{t_{n+1},X^{n+1}})]_{t_{n+1},t_{n+2}}\big] \\
=&\!\!\!\! \sum\limits_{\alpha\in \mathcal{A}_3}  \sum\limits_{i=1}^d \mathbb{E}_{t_{n+1}}^{X^{n+1}}\big[I_{(0,0,0)}\big[D_{s_1s_2s_3}^\alpha\{F_{x_i}^{X^{n+1}}\} g_{\alpha}^i(s_1,X_{s_1}^{t_{n+1},X^{n+1}})\big]_{t_{n+1},t_{n+2}}\big].
\end{array}\end{equation*}
By the integration-by-parts of Malliavin calculus \eqref{Malp} again we deduce
\begin{equation*}
\begin{array}{rl}
&\!\!\!\! \left|\mathbb{E}_{t_{n}}^{X^{n}}\left[\Delta W_{n}U_{n+1}^{n+2}\right]\right|\\
=& \!\!\!\!
\Big|\sum\limits_{\alpha\in \mathcal{A}_3}\sum\limits_{i=1}^d\mathbb{E}_{t_{n}}^{X^{n}}\Big[\Delta W_{n}\mathbb{E}_{t_{n+1}}^{X^{n+1}}\Big[I_{(0,0,0)}\big[D_{s_1s_2s_3}^\alpha\{F_{x_i}^{X^{n+1}}\} g_{\alpha}^i(s_1,X_{s_1}^{t_{n+1},X^{n+1}})\big]_{t_{n+1},t_{n+2}}\Big]\Big]\Big|\\
=&\!\!\!\! \Big|\sum\limits_{\alpha\in \mathcal{A}_3}\sum\limits_{i=1}^d\mathbb{E}_{t_{n}}^{X^{n}}\Big[I_{(0,0,0)}\big[(W_{s_1}
 -W_{t_n})D_{s_1s_2s_3}^\alpha\{F_{x_i}^{X^{n+1}}\} g_{\alpha}^i(s_1,X_{s_1}^{t_{n+1},X^{n+1}})\big]_{t_{n+1},t_{n+2}}\Big]\Big|\\
=&\!\!\!\!
\Big|\sum\limits_{\alpha\in \mathcal{A}_3}\sum\limits_{i=1}^d\int_{t_{n+1}}^{t_{n+2}}\int_{t_{n+1}}^{s_3}\int_{t_{n+1}}^{s_2}\int_{t_{n}}^{s_1} \mathbb{E}_{t_{n}}^{X^{n}}\big[D_s \big\{D_{s_1 s_2 s_3}^{\alpha}\{F_{x_i}^{X^{n+1}}\}
 g_{\alpha}^i(s_1,X_{s_1}^{t_{n+1},X^{n+1}})\big\}\big]ds ds_1ds_2
ds_3\Big|\\
\leq &\!\!\!\! C(1+|X^n|^q)\Delta^4,
\end{array}
\end{equation*}
which proves \eqref{lemleq4}.
The estimate \eqref{lemqel} can be similarly proved.
The proof is complete.
\end{proof}



\begin{lem}\label{lem55}
For $X^{n+1}=\sum\limits_{\alpha\in \Gamma_2}g_\alpha(t_n,X^n)I_{\alpha, n}$,
if $b,\sigma\in C_b^{3,6}$, $H\in C_b^5$,
then under Hypotheses \ref{hyp0},
for $1\leq n \leq N-2$, there exists a generic positive integer $q$ such that
\begin{equation}\label{312}
\big|U_{n}^{n+1}-\mathbb{E}_{t_{n}}^{X^{n}}\big[U_{n+1}^{n+2}\big]\big|
\leq  C(1+|X^n|^q) \Delta^4
\end{equation}
for $1\leq n \leq N-2$, where
$U_{n}^{n+1}=\mathbb{E}_{t_{n}}^{X^{n}}\big[H(X_{t_{n+1}}^{t_{n},X^{n}})-H(X^{n+1})\big]$,
and $C$ is a positive constant depending on $K$, and upper bounds of
the derivatives of $b,\sigma$ and $H$.
\end{lem}

\begin{proof}
By the multiple Taylor expansion and the definition of $U_{n}^{n+1}$, we know
\begin{equation}\label{312a}
\begin{aligned}
\big|U_{n}^{n+1}-\mathbb{E}_{t_n}^{X^n}\big[U_{n+1}^{n+2}\big]\big|&=
\big|\mathbb{E}_{t_n}^{X^n}\big[H(X_{t_{n+1}}^{t_n,X^n})-H(X^{n+1})
-\mathbb{E}_{t_{n+1}}^{X^{n+1}}[H(X_{t_{n+2}}^{t_{n+1},X^{n+1}})-H(X^{n+2})]\big]\big|\\
&=\sum\limits_{i=1}^d \mathbb{E}_{t_{n}}^{X^{n}}[h_{n}^i F_{x_{i}}^{X^n}]
-\sum\limits_{i=1}^d \mathbb{E}_{t_{n+1}}^{X^{n+1}}[h_{n+1}^i F_{x_{i}}^{X^n+1}],
\end{aligned}
\end{equation}
where $h_{n}^i=\sum\limits_{\alpha\in \mathcal{A}_3}I_\alpha[g_{\alpha}^i(\cdot,X_\cdot^{t_{n},X^{n}})]_{t_{n},t_{n+1}}$
and $F_{x_{i}}^{X^n}=\int_0^1 H_{x_i}'\big(X^{n+1}+\lambda(X_{t_{n+1}}^{t_n,X^n}-X^{n+1})\big)d\lambda.$
Thus, by the integration-by-parts formula \eqref{Malp} of Malliavin calculus, we have
\begin{equation*}\label{3.57a}
\begin{array}{rll}
\sum\limits_{i=1}^d \mathbb{E}_{t_n}^{X^n}[F_{x_i}^{X^n} h^{n}_i]
=&\!\!\!\!
\sum\limits_{i=1}^d\mathbb{E}_{t_n}^{X^n}\Big[\sum\limits_{\alpha\in \mathcal{A}_3} F_{x_i}^{X^{n}}I_{\alpha,n}\, g_{\alpha}^i(t_n,X^n)+\sum\limits_{\alpha\in \mathcal{A}_4} F_{x_i}^{X^{n}}I_\alpha[g_{\alpha}^i(\cdot,X_\cdot^{t_n,X^n})]_{t_n,t_{n+1}}\Big] \\
=&\!\!\!\!
\sum\limits_{i=1}^d\sum\limits_{\alpha\in \mathcal{A}_3}I_{(0,0,0)}\Big[\mathbb{E}_{t_n}^{X^n}\big[D_{s_1s_2s_3}^\alpha
  \{F_{x_i}^{X^{n}}\}\big]\Big]_{t_n,t_{n+1}}g_{\alpha}^i(t_n,X^n)
\\&\!\!\!\!
+\sum\limits_{i=1}^d\sum\limits_{\alpha\in \mathcal{A}_4}I_{(0,0,0,0)}\Big[\mathbb{E}_{t_n}^{X^n}\big[D_{s_1s_2s_3s_4}^\alpha \{F_{x_i}^{X^{n}}\}g_{\alpha}^i(s_1,X_{s_1}^{t_n,X^n})\big]\Big]_{t_n,t_{n+1}}.
\end{array}
\end{equation*}
For $\lambda\in (0,1]$, we assume $\psi_n=\big(\psi_{n,1},\psi_{n,2},\ldots,\psi_{n,d}\big)
:=X^{n+1}-X^n+\lambda\big(X_{t_{n+1}}^{t_n,X^n}-X^{n+1}\big)$ with its $i$-th component
 $\psi_{n,i}=\phi_{n}^i+\lambda h_{n}^i$ ($1\leq i\leq d$), then
by the Taylor expansion, we deduce
\begin{equation*}
\begin{array}{rl}
F_{x_i}^{X^n}=&\!\!\!\!\int_0^1H_{x_i}'\big(X^{n+1}+\lambda(X_{t_{n+1}}^{t_n,X^n}-X^{n+1})\big)d\lambda= \int_0^1H_{x_i}'\big(X^{n}+\psi_n\big)d\lambda\\
 = & \!\!\!\! H_{x_i}'(X_n)+\int_0^1\sum\limits_{j=1}^4\frac{1}{j!}\Big(\psi_{n,1}\frac{\partial}{\partial x_1}+\psi_{n,2}\frac{\partial }{\partial x_2}+\cdots+\psi_{n,d}\frac{\partial }{\partial x_d}\Big)^j H_{x_i}'(X^n)d\lambda+R_5,
\end{array}
\end{equation*}
where
\begin{equation*}\label{51:44}
\ds \Big(\psi_{n,1}\frac{\partial}{\partial x_1}
+\psi_{n,2}\frac{\partial }{\partial x_2}
+\cdots+\psi_{n,d}\frac{\partial}{\partial x_d}\Big)^j
=\sum\limits_{r_1+r_2+\cdots+r_d=j}\frac{j!}{r_1!r_2!\cdots r_d!}\,
\psi_{n,1}^{r_1}\psi_{n,2}^{r_2}\cdots \psi_{n,d}^{r_d}
\frac{\partial^j}{\partial x_1^{r_1}\partial x_2^{r_2}\cdots \partial x_d^{r_d}},
\end{equation*}
and
\begin{equation*}
R_{5}=\frac{1}{5!}\int_0^1\int_0^1\Big(\psi_{n,1}\frac{\partial}{\partial x_1}
+\psi_{n,2}\frac{\partial }{\partial x_2}+\cdots
+\psi_{n,d}\frac{\partial }{\partial x_d}\Big)^5
H_{x_i}'\big(X^{n}+\mu\psi_n\big) d\mu d\lambda.
\end{equation*}
Further,
taking the Malliavin derivative $D_{s_1s_2s_3}^{\alpha}$ to $F_{x_i}^{X^{n}}$
with $\alpha=(j_1,j_2,j_3)\in \mathcal{A}_3$ implies
\begin{equation*}
\begin{aligned}
&
\mathbb{E}_{t_n}^{X^n}\big[D_{s_1s_2s_3}^\alpha(F_{x_i}^{X^{n}})\big]
= \int_0^1\mathbb{E}_{t_n}^{X^n}\big[D_{s_1s_2s_3}^{\alpha}H_{x_i}'\big(X^n+\psi_{n}\big)\big]d\lambda\\
=\;&
D_{s_1s_2s_3}^\alpha H_{x_i}'(X^n)+
\int_0^1\sum_{j=1}^4\frac{1}{j!}\mathbb{E}_{t_n}^{X^n}\big[D_{s_1s_2s_3}^{\alpha}
\big\{\big(\psi_{n,1}\frac{\partial}{\partial x_1}+\psi_{n,2}\frac{\partial }{\partial x_2}+\cdots+\psi_{n,d}\frac{\partial }{\partial x_d}\big)^j H_{x_i}'(X^n)\big\}\big]d\lambda\\
&
+\mathbb{E}_{t_n}^{X^n}\big[D_{s_1s_2s_3}^{\alpha}(R_5)\big]\\
=&
\int_0^1\sum_{j=1}^4\sum\limits_{r_1+r_2+\cdots+r_d=j}\frac{1}{r_1!r_2!\cdots r_d!}
\mathbb{E}_{t_n}^{X^n}\big[D_{s_1s_2s_3}^{\alpha}\big\{\psi_{n,1}^{r_1}\psi_{n,2}^{r_2}\cdots \psi_{n,d}^{r_d}\big\}\big]\frac{\partial^{j+1} H}{\partial x_i^{} \partial x_1^{r_1}\partial x_2^{r_2}\cdots \partial x_d^{r_d}}(X^n)d\lambda\\
&
+D_{s_1s_2s_3}^\alpha H_{x_i}'(X^n)+\mathbb{E}_{t_n}^{X^n}\big[D_{s_1s_2s_3}^{\alpha}(R_5)\big].
\end{aligned}
\end{equation*}
If 
$b,\,\sigma\in C_b^{3,6}$, $H\in C_b^5$, using \eqref{phin} we have
\begin{equation*}\label{3.60}
\begin{aligned}
\mathbb{E}_{t_n}^{X^n}\big[D_{s_1s_2s_3}^\alpha
F_{x_i}^{X^{n}}\big]
 =\omega_\alpha(t_n,X^n)+O_{n}(\Delta ),
\end{aligned}
\end{equation*}
where $\omega_\alpha(t_n,X^n)$ is a function depending only on the index $\alpha$,
$t_n$, $b_i(t_n,X^n)$, $\sigma_{ij}(t_n,X^n)$, $H_{x_i}'(X^n)$ ($1\leq i,j\leq d$),
and their derivatives;
the notation $O_{n}(\Delta)$ means  that it has the estimate
$|O_{n}(\Delta)|\leq C (1+|X^n|^q) \Delta$ with a prior  known integer $q$
 which does not depend on $n$.
From $I_{(0,0,0), n}
=\int_{t_n}^{t_{n+1}}\int_{t_n}^{s_3}\int_{t_n}^{s_2}ds_1\,ds_2\,ds_3=\frac{1}{6}\Delta^3$,
we obtain
\begin{equation*}
\begin{aligned}
&\sum\limits_{i=1}^d\sum_{\alpha\in \mathcal{A}_3}I_{(0,0,0)}\Big[\mathbb{E}_{t_n}^{X^n}\big[D_{s_1s_2s_3}^\alpha
\{F_{x_i}^{X^{n}}\}\big]\Big]_{t_n,t_{n+1}}g_{\alpha}^i(t_n,X^n)\\
=\;&
\sum\limits_{i=1}^d\sum_{\alpha\in \mathcal{A}_3} \big(\omega_\alpha(t_n,X^n)+O_n(\Delta)\big)
g_{\alpha}^i(t_n,X^n)I_{(0,0,0), n}\\
=\;&
\frac{1}{6} \Delta^3\sum\limits_{i=1}^d\sum_{\alpha\in\mathcal{A}_3}\omega_\alpha(t_n,X^n)
g_{\alpha}^i(t_n,X^n)+O_n(\Delta^4).
\end{aligned}
\end{equation*}
Then under the assumptions of the lemma, by the inequality \eqref{c5Ex} in
Hypothesis \ref{hyp3}, it holds
\begin{equation*}\label{3.61}
\begin{aligned}
& \frac{1}{6}\Delta^3\sum\limits_{i=1}^d\sum_{\alpha\in \mathcal{A}_3}\Big|\mathbb{E}_{t_n}^{X^n}\Big[\omega_\alpha(t_n,X^n)g_{\alpha}^i(t_n,X^n)-\omega_\alpha(t_{n+1}, X^{n+1})g_{\alpha}^i\big(t_{n+1},X^{n+1}\big)\Big]\Big|\\
& \qquad \leq  C(1+|X^n|^q)\Delta^4.
\end{aligned}
\end{equation*}
Under Hypothesis \ref{hyp3}, from the equations \eqref{312a} and the above
inequality  we obtain
\begin{equation}\label{lem41}
\begin{array}{rl}
&\Big|\sum\limits_{i=1}^d\sum\limits_{\alpha\in \mathcal{A}_3}\Big\{I_{(0,0,0)}\Big[\mathbb{E}_{t_n}^{X^n}\big[D_{s_1s_2s_3}^\alpha
\{F_{x_i}^{X^{n}}\}\big]\Big]_{t_n,t_{n+1}}g_\alpha^i(t_n,X^n)\\
& \qquad\qquad\quad -I_{(0,0,0)}\Big[\mathbb{E}_{t_{n}}^{X^{n}}\big[D_{s_1s_2s_3}^\alpha
\{F_{x_i}^{X^{n+1}}\}\big]\Big]_{t_{n+1},t_{n+2}}g_\alpha^i(t_{n+1},X^{n+1})\Big\}\Big|\\
& \qquad  \leq  C(1+|X^n|^q)\Delta^4.
\end{array}
\end{equation}
And under the assumption,  it holds that
\begin{equation}\label{lem42}
\begin{aligned}
&\Big|\sum\limits_{i=1}^d\sum\limits_{\alpha\in \mathcal{A}_4}
I_{(0,0,0,0)}\Big[\mathbb{E}_{t_n}^{X^n}\big[D_{s_1s_2s_3s_4}^\alpha \{F_{x_i}^{X_{n}}\}g_{\alpha}^i(s_1,X_{s_1}^{t_n,X^n})\big]\Big]_{t_n,t_{n+1}}\Big|\\
&\qquad \leq C(1+|X^n|^q)\Delta^4,\\
& \Big|\sum\limits_{i=1}^d\sum\limits_{\alpha\in \mathcal{A}_4}
I_{(0,0,0,0)}\Big[\mathbb{E}_{t_n}^{X^n}\big[D_{s_1s_2s_3s_4}^\alpha\{F_{x_i}^{X^{n+1}}\}
g_{\alpha}^i(s_1,X_{s_1}^{t_{n+1},X^{n+1}})\big]\Big]_{t_{n+1},t_{n+2}}\Big|\\
&\qquad \leq C(1+|X^n|^q)\Delta^4.
\end{aligned}
\end{equation}
Now combining the estimates \eqref{312a}, \eqref{lem41} and \eqref{lem42},
we complete the proof.
\end{proof}

\begin{lem}\label{lem56} (See \cite{ZWP09})
Let $\big(X_r^{t,x},Y_r^{t,x},Z_r^{t,x}\big)_{t\leq r\leq T}$ be the solution of \eqref{DFBSDEs2}, and let
$R_{y1}^n$ and $R_{z1}^n$ be the truncation errors defined in
\eqref{s2.3}--\eqref{s2e13} for the C-N scheme.
 If the terminal function  $\varphi\in C_b^{4+\alpha}$ for some $\alpha\in (0,1)$, $b,\sigma$ are bounded, $b,\sigma\in C_b^{2,4}$, and $f\in
C_b^{2,4,4,4}$, then it holds that
\begin{equation}
\begin{aligned}
&\mathbb{E}[|R_{y1}^{N-1}|^2]\leq C(t_N-t_{N-1})^4 = C\Delta^8, \quad
\mathbb{E}[|R_{z1}^{N-1}|^2]\leq C(t_N-t_{N-1})^4= C\Delta^8,\\&
\mathbb{E}[|R_{y1}^{n}|^2] \leq C\Delta^6, \quad
\mathbb{E}[|R_{z1}^n|^2] \leq C\Delta^6, \quad 0\leq n\leq N-2.
\end{aligned}
\end{equation}
And if $\varphi\in C_b^{5+\alpha}$ for some $\alpha\in (0,1)$, $b,\sigma\in C_b^{2,5}$ and $f\in C_b^{2,5,5,5}$, then it holds that
\begin{equation}
\mathbb{E}[|\nabla_{x^n} R_{y1}^{n}|^2] \leq C\Delta^6, \quad
\mathbb{E}[|\nabla_{x^n} R_{z1}^n|^2] \leq C\Delta^6, \quad 0\leq n\leq N-2.
\end{equation}
Here $C$ is a generic positive constant depending on $K$, the initial condition of $X_t$,
and upper bounds of the derivatives of $b$, $\sigma$, $f$ and $\varphi$.
\end{lem}

The above lemma can be proved by using the Taylor and It\^o-Taylor expansion.
Here we omit the proof. Please see the details in \cite{ZWP09}.
Now combining Lemmas \ref{lem51}--\ref{lem55}, we  state our truncation error
estimates in the following lemma.
\begin{lem}\label{lem57}
For $X^{n+1}=\sum\limits_{\alpha\in \Gamma_2}g_\alpha(t_n,X^n)I_{\alpha, n}$,
if $b,\sigma\in C_b^{2,5}$, $f(t,X,Y,Z)\in C_b^{3,5,5,5}$ and
$\varphi\in C_b^{6+\alpha}$ for some $\alpha\in (0,1)$, then
under Hypotheses \ref{hyp0} and \ref{hyp3},
there exists a generic positive integer $q$ such that
\begin{equation}\label{ey1}\begin{aligned}&
\max_{0 \leq n\leq N-2}|\nabla_{x^n} R_{y2}^{n}| \leq C(1+|X^n|^q)\Delta^3, \\&
\max_{0 \leq n\leq N-2}|\nabla_{x^n} R_{z2}^{n}| \leq C(1+|X^n|^q)\Delta^3,
\end{aligned}\end{equation}
and
\begin{equation}\label{ey2}
\begin{array}{ll}
&\max\limits_{0 \leq n\leq N-2}|\mathbb{E}_{t_n}^{X^n}[R_{y1}^{n+1}\Delta W_{n}^\top]|
\leq C(1+|X^n|^q)\Delta^4, \\
&\max\limits_{0 \leq n\leq N-2}|\mathbb{E}_{t_n}^{X^n}[\Delta W_{n} \nabla_{x^n} R_{y1}^{n+1}]|
\leq C(1+|X^n|^q)\Delta^4, \\&
\max\limits_{0 \leq n\leq N-2}|\mathbb{E}_{t_n}^{X^n}[R_{y2}^{n+1}\Delta W_{n}^\top]| \leq C(1+|X^n|^q)\Delta^4, \\
&\max\limits_{0 \leq n\leq N-2}|\mathbb{E}_{t_n}^{X^n}[\Delta W_{n} \nabla_{x^n} R_{y2}^{n+1}]|
\leq C(1+|X^n|^q)\Delta^4 .
\end{array}
\end{equation}
And if 
$b,\sigma\in C_b^{2,5}$, $f(t,X,Y,Z)\in C_b^{3,6,6,6}$ and
$\varphi\in C_b^{7+\alpha}$ for some $\alpha\in (0,1)$, then
\begin{equation}\label{ez}
\begin{array}{rl}
&\max\limits_{0 \leq n\leq N-2}|R_{z1}^n-\mathbb{E}^{X^n}_{t_n}[R_{z1}^{n+1}]|
\leq C(1+|X^n|^q) \Delta^4,\quad\\
&\max\limits_{0 \leq n\leq N-2}|R_{z2}^n-\mathbb{E}^{X^n}_{t_n}[R_{z2}^{n+1}]|
\leq C(1+|X^n|^q) \Delta^4.
\end{array}
\end{equation}
Here $C$ is a generic positive constant depending on $K$, and upper bounds of
the derivatives of $b$, $\sigma$, $f$ and $\varphi$.
\end{lem}
\begin{proof}
Under the conditions of the lemma,  by the Feynman-Kac formula \eqref{s2:e2}, the solution $(Y_t^{t_n,X^n}, Z_t^{t_n,X^n})$ of FBSDEs \eqref{DFBSDEs2} can be represented as
\begin{equation}\label{335}
Y_t^{t_n,X^n}=u(t,X_t^{t_n,X^n}),\quad
Z_t^{t_n,X^n}= u_x(t,X_t^{t_n,X^n})\sigma(t,X_t^{t_n,X^n}),\,
\forall\; t\in [0,T),
\end{equation}
where $u(t,x)$ satisfies the parabolic PDE \eqref{PDEs}.
According to (\ref{335}), we set
\begin{equation*}
\begin{array}{rl}
H(X_{t_{n+1}}^{t_n,X^n})
=&\!\!\!\!
\mathbb{E}_{t_n}^{X^n}[Y_{t_{n+1}}^{t_n,X^n}]+\frac{1}{2}\Delta \mathbb{E}_{t_n}^{X^n}[f_{t_{n+1}}^{t_n,X^n}]\\
=&\!\!\!\!
\mathbb{E}_{t_n}^{X^n}[u(t_{n+1},X_{t_{n+1}}^{t_n,X^n})]
\\&\!\!\!\!
+\frac{1}{2}\Delta \mathbb{E}_{t_n}^{X^n}\big[f\big(t_{n+1},X_{t_{n+1}}^{t_n,X^n}, u(t_{n+1},X_{t_{n+1}}^{t_n,X^n}),u_x(t_{n+1},
X_{t_{n+1}}^{t_n,X^n})\sigma(t_{n+1},X_{t_{n+1}}^{t_n,X^n})\big)\big]
\end{array}
\end{equation*}
and
\begin{equation*}
\begin{array}{rl}
H(X^{n+1})
=&\!\!\!\!
\mathbb{E}_{t_n}^{X^n}[Y_{t_{n+1}}^{t_{n+1},X^{n+1}}]+\frac{1}{2}\Delta \mathbb{E}_{t_n}^{X^n}[f_{t_{n+1}}^{t_{n+1},X^{n+1}}]\\
=&\!\!\!\!
\mathbb{E}_{t_n}^{X^n}[u(t_{n+1},X^{n+1})]
\\&\!\!\!\!
+\frac{1}{2}\Delta \mathbb{E}_{t_n}^{X^n}\big[f\big(t_{n+1},X^{n+1}, u(t_{n+1},X^{n+1}),u_x(t_{n+1},X^{n+1})\sigma(t_{n+1},X^{n+1})\big)\big].
\end{array}
\end{equation*}
Then, $$R_{y2}^n = H(X_{t_{n+1}}^{t_n,X^n})-H(X^{n+1}).$$
By the theory of partial differential equations \cite{Evan98},
under the conditions of the lemma,
it is easy to check that the function $H$ satisfies the conditions in
Lemma \ref{lem51}, thus we have the estimate $\nabla_{x^n}R_{y2}^n$ in \eqref{ey1}.
Similarly under the conditions of this lemma, we have the estimates in \eqref{ey2}
and (\ref{ez}) by using Lemma \ref{lem52}--\ref{lem55}. The proof is completed.
\end{proof}

\begin{lem}\label{lem-ini21}
Assume $Y^N=\varphi(X^N)$. Under Hypothesis \ref{hyp3}
and the conditions of Lemma \ref{lem56}, it holds that
\begin{equation}\label{inires}
\mathbb{E}[|e_{Y}^{N-1}|^2]\leq C\Delta^4,\quad
\mathbb{E}[|e_{Z}^{N-1}|^2]\leq C\Delta^{4},\quad
\mathbb{E}[|e_{\nabla Y}^{N-1}|^2]\leq C\Delta^4,\quad
\mathbb{E}[|e_{\nabla Z}^{N-1}|^2]\leq C\Delta^4.
\end{equation}
where $C$ is a generic constant depending on $K,L$, the initial condition of $X_t$,
and upper bounds of derivatives of $b,\sigma, f$ and $\varphi$.
\end{lem}
\begin{proof}
We know that for weak order-2 scheme \eqref{Xnn10}, Hypothesis \ref{hyp3} holds true with
$\beta = \gamma = 2$ (readers can refer to Section 4.3.2 in \cite{ZZJ14} for the detailed proof).
Combining with Lemma \ref{lem56}, we get
\begin{equation}\label{s3:e26c2}
\sum\limits_{j=1}^2 \mathbb{E}\big[|R_{yj}^{N-1}|^2\big]\leq C\Delta^8, \qquad
\sum\limits_{j=1}^2\mathbb{E}\big[|R_{zj}^{N-1}|^2\big]\leq C\Delta^8.
\end{equation}
By the equalities \eqref{s2.3}, \eqref{ZN-1} and \eqref{scheme-step1}, we deduce
$$
\mathbb{E}[|e_Z^{N-1}|^2] =
\frac{1}{\Delta^2}\mathbb{E}[|\sum\limits_{j=1}^2R_{zj}^{N-1}|^2]\leq
\frac{1}{\Delta^2}\sum\limits_{j=1}^2\mathbb{E}[|R_{zj}^{N-1}|^2]
\leq  C\Delta^{4},
$$
and
\begin{equation}\nonumber
\begin{array}{rl}
\mathbb{E}[|e_Y^{N-1}|^2] \leq
C\Delta^2\mathbb{E}[|e_Y^{N-1}|^2+|e_Z^{N-1}|^2]+C\sum\limits_{j=1}^2\mathbb{E}[|R_{yj}^{N-1}|^2]
\leq  C\Delta^2 \mathbb{E}[|e_Y^{N-1}|^2] + C\Delta^4,
\end{array}
\end{equation}
which implies $\mathbb{E}[|e_Y^{N-1}|^2]\leq \frac{C\Delta^4}{1-C\Delta^4}\leq  C\Delta^4$.
Similarly we can prove $\mathbb{E}[|e_{\nabla Y}^{N-1}|^2]\leq C\Delta^4$ and
$\mathbb{E}[|e_{\nabla Z}^{N-1}|^2]\leq C\Delta^{4}$.
The proof is completed.
\end{proof}

\subsubsection{Proof of Theorem \ref{thm2}}
After the above preparations, we now give the proof of Theorem \ref{thm2} as follows.
\begin{proof}
Under the conditions of the theorem, if Hypothesis \ref{hyp3} holds,
we have $$\mathbb{E}[|X^n|^q] \leq C(1+\mathbb{E}[|X_0|^q]),$$
then according to Lemmas \ref{lem56}--\ref{lem-ini21},
we obtain the estimates
\begin{equation*}
\begin{array}{rl}
&\hspace{2cm}
\mathbb{E}[|e_{Y}^{N-1}|^2+|e_{Z}^{N-1}|^2+|e_{\nabla Y}^{N-1}|^2
 +|e_{\nabla Z}^{N-1}|^2]\leq C\Delta^{4}, \\ &
\max\limits_{0 \leq n\leq N-2}\sum\limits_{j=1}^2
\mathbb{E}\Big[|R_{yj}^{n}|^2+|\nabla_{x^n} R_{yj}^{n}|^2+|\nabla_{x^n} R_{zj}^{n}|^2
+|\mathbb{E}^{X^n}_{t_n}[R_{yj}^{n+1}]|^2
+|\mathbb{E}^{X^n}_{t_n}[R_{zj}^{n+1}]|^2\\
&\hspace{2cm}
+|\mathbb{E}^{X^n}_{t_n}[\nabla_{x^n}R_{yj}^{n+1}]|^2
+|\mathbb{E}^{X^n}_{t_n}[\nabla_{x^n}R_{zj}^{n+1}]|^2
+|\mathbb{E}^{X^n}_{t_n}[\nabla_{x^{n+1}}R_{zj}^{n+1}]|^2
\Big]\leq C\Delta^6,\\
&
\max\limits_{0 \leq n\leq N-2}\sum\limits_{j=1}^2
\mathbb{E}\Big[|\mathbb{E}^{X^n}_{t_n}[\Delta W_n\nabla_{x^n}R_{yj}^{n+1}]|^2
+|\mathbb{E}^{X^n}_{t_n}[R_{yj}^{n+1}\Delta W_{n}^\top]|^2
+|R_{zj}^n-\mathbb{E}^{X^n}_{t_n}[R_{zj}^{n+1}]|^2
\Big]\\&
\leq C\Delta^8.
\end{array}
\end{equation*}
Inserting the above estimates into \eqref{eq:thm1} in Theorem \ref{thm1},
we easily deduce \eqref{eq:thm2}.
The proof is completed.
\end{proof}


\section{Conclusions}

In this paper, we considered the theoretical error estimates
of the Crank-Nicolson (C-N) scheme for solving decoupled FBSDEs
proposed in \cite{ZLF14}.
By properly using the Young's inequality to the error equations of the C-N scheme
and their associated variational equations,
we first rigorously obtained a general error estimate result for the C-N scheme.
This result also implies the stability of the scheme.
Then by the Taylor and It\^o-Taylor expansions, the theory of multiple Malliavin calculus,
and the local truncation error cancelation techniques,
we theoretically obtained the truncation error estimates of the scheme.
Finally based on the general error estimate result
and the truncation error estimates,
we theoretically proved
that the accuracy of the C-N scheme for solving decoupled FBSDEs
is of second order.

%

%
%
%
%
%
%
%
%
%
%
%
%
%
%
%
%
%
%

\Acknowledgements{This author's research is partially supported by  the National
Natural Science Foundations of China under grant numbers 11426152, 11501366 and 11571206.}


\end{document}